\documentclass[11pt,reqno]{amsart}    
 
\usepackage[margin=1in]{geometry}
\usepackage{amsfonts,amsmath,amssymb,enumerate}    
\usepackage{amsthm, bm}    
\usepackage{mathrsfs}
\usepackage{setspace}
\usepackage{comment}
\onehalfspace
\usepackage{tikz}

\usetikzlibrary{arrows,matrix}
\usetikzlibrary{decorations.markings,shapes.geometric,shapes.misc}
\usetikzlibrary{decorations.pathreplacing}  
 
\tikzset{cross/.style={cross out, draw=black, minimum size=2*(#1-\pgflinewidth), inner sep=0pt, outer sep=0pt}, cross/.default={1pt}}

\numberwithin{equation}{section}
\usepackage{mathtools}
\theoremstyle{plain}
\newtheorem{thm}{Theorem}[section]
\newtheorem{lem}[thm]{Lemma}
\newtheorem{prop}[thm]{Proposition}
\newtheorem{cor}[thm]{Corollary}

\newtheorem*{prop*}{Proposition}
\theoremstyle{definition}

\newtheorem{exmp}[thm]{Example}
\theoremstyle{remark}
\newtheorem{rem}[thm]{Remark}

\newtheorem*{rem*}{Remark}

\newcommand{\be}{\begin{equation}}    
\newcommand{\ee}{\end{equation}}    
\newcommand{\beu}{\begin{equation*}}    
\newcommand{\eeu}{\end{equation*}}    
\newcommand{\bea}{\begin{eqnarray}}    
\newcommand{\eea}{\end{eqnarray}}    
\newcommand{\beaa}{\begin{eqnarray*}}    
\newcommand{\eeaa}{\end{eqnarray*}}    
\newcommand{\bmx}{\begin{pmatrix}}    
\newcommand{\emx}{\end{pmatrix}}

\newcommand{\ul}{\underline}    
\newcommand{\del}{\partial}    
\newcommand{\g}{{\mathfrak g}}

\newcommand{\n}{{\mathfrak n}}    
\newcommand{\h}{{\mathfrak h}}

\newcommand{\p}{{\mathfrak p}}

\newcommand{\B}{\mathcal B}

\newcommand{\Heis}{H}

\newcommand{\mf}{\mathfrak}
\newcommand{\mc}{\mathcal}    

\newcommand{\gh}{{\widehat \g}}    

\newcommand{\hh}{{\widehat \h}}
\newcommand{\nG}{{\n_{\langle G_i\rangle}}}
\newcommand{\gs}{{\g^{\sigma}}}

\newcommand{\D}{\mathcal D}

\newcommand{\half}{\frac{1}{2}}

\newcommand{\nn}{\nonumber}

\newcommand{\8}{{\infty}}

\newcommand{\eps}{\epsilon}    
\newcommand{\tr}{\,{\rm tr}}    
    
\newcommand{\rank}{{\rm rank}}

\newcommand{\ad}{{\rm ad}}

\newcommand{\Z}{{\mathbb Z}}

\newcommand{\C}{{\mathbb C}}

\newcommand{\bra}[1]{{\,\left<#1\right|}\,}    
\newcommand{\ket}[1]{{\left|#1\right>}\,}

\newcommand{\id}{{\mathrm{id}}}

\newcommand{\A}[2]{A_{#1,q^{#2}}}

\newcommand{\goi}[2]{=}    
\newcommand{\Hom}{\mathrm{Hom}}

\newcommand{\on}{.}    

\newcommand{\Cx}{\mathbb C^\times}

\renewcommand{\binom}[2]{{#1 \brack #2}}

\newcommand{\btp}{\begin{tikzpicture}[baseline=0pt,scale=0.9,line width=0.25pt]}    
\newcommand{\etp}{\end{tikzpicture}}    
\newcommand{\Ron}{\color{blue}}
\newcommand{\Roff}{\color{black}}

\newcommand{\path}{\longrightarrow}

\newcommand{\ha}{\mbox{\small $\frac{1}{2}$}}
\newcommand{\atp}[2]{\underset{\substack{$ $ \\ \tikz{
\draw[->] (0,0) -- (0,.13);} \\[-.2mm] {#1}}}{\smash{#2}}}

\DeclareMathOperator{\res}{res}

\newcommand{\nord}[1]{:\!\! #1 \!\!:}

\newcommand*{\longhookrightarrow}{\ensuremath{\lhook\joinrel\relbar\joinrel\rightarrow}}

\newcommand{\VV}{{\mathbb V}}

\newcommand{\Vcrit}{{\mathbb V_0^{-h^\vee}}}
\newcommand{\WW}{{\mathbb W}}
\newcommand{\MM}{{\mathbb M}}
\newcommand{\M}{\mathcal M}

\newcommand{\vac}{v_0}
\newcommand{\wac}{\ket{}\!}
\newcommand{\lwac}{\bra{}\!}
\newcommand{\resd}{\vee}
\newcommand{\pdd}{{{\frac{1}{(n-1)!}\frac{\del^{n-1}}{\del u^{n-1}}}}}
\newcommand{\lsigma}{L_{\sigma}}
\newcommand{\CP}{\mathbb C P}

\newcommand{\Gaud}{\mathscr Z}
\newcommand{\Mh}{\mathsf M} 
\newcommand{\ie}{\textit{i.e. }}
\newcommand{\la}{\big\langle}
\newcommand{\ra}{\big\rangle}
\newcommand{\zs}{\bm z}
\newcommand{\xs}{\bm x}
\newcommand{\ns}{\bm n}
\newcommand{\nsa}{{n_\8,\ns,n_0}}
\newcommand{\zsa}{{\8,\zs,0}}
\newcommand{\xsa}{{\8,\xs,0}}

\newcommand{\gG}{\g_{\8,\zs,0}^\Gamma}
\newcommand{\gN}{\gh_{\8,N,0}}

\newcommand{\gNpp}{\gh_{\8,N,0}^{+\!\!+}}

\newcommand{\gGu}{\g_{\8,\zs,u,0}^\Gamma}
\newcommand{\gNu}{\gh_{\8,N+1,0}}

\newcommand{\Upp}{U(\gNpp)}

\newcommand{\MMx}{\MM}
\newcommand{\Mx}{\M}
\newcommand{\cai}{\varphi}

\newcommand{\vla}{\mathscr L}
\newcommand{\vma}{\mathscr M}

\newcommand{\ueva}{\VV(\vla)}
\newcommand{\ueval}{\VV(\vma)}

\newcommand{\I}{\mathscr I}

\newcommand{\cent}{\mathsf c} 
\newcommand{\op}{\mathsf{op}}

\newcommand{\into}{\hookrightarrow}
\newcommand{\onto}{\twoheadrightarrow}
\newcommand{\longinto}{\lhook\joinrel\relbar\joinrel\rightarrow}
\newcommand{\longonto}{\relbar\joinrel\twoheadrightarrow}

\newcommand{\tak}[1]{\mathcal T_{#1}}
\newcommand{\ox}{\otimes}
\newcommand{\ssc}{\mf z}
\newcommand{\ZT}{\Z/T\Z}
\newcommand{\dell}[2]{\del^{(#1)}_{#2}}

\newcommand{\Achi}{\mc A_{\8,\zs,0}^{1,(1),1}(\g,\sigma,\chi)^\Gamma}
\newcommand{\U}{\mathsf L}

\DeclareMathOperator{\Ind}{Ind}
\DeclareMathOperator{\Coind}{Coind}
\DeclareMathOperator{\End}{End}

\DeclareMathOperator{\Vect}{Vect}

\DeclareMathOperator{\Weyl}{Weyl}
\DeclareMathOperator{\Der}{Der}
\DeclareMathOperator{\Homres}{Hom}
\DeclareMathOperator{\Lie}{Lie}

\author{Beno\^{\i}t Vicedo} 
\author{Charles Young}
\address{\vspace{-.15cm} 
School of Physics, Astronomy and Mathematics, University of Hertfordshire, College Lane, Hatfield AL10 9AB, UK.}  
\email{benoit.vicedo@gmail.com}  
\email{c.a.s.young@gmail.com}

\begin{document} 
\title{Cyclotomic Gaudin models with irregular singularities}

\begin{abstract} 
Generalizing the construction of the cyclotomic Gaudin algebra from \cite{VY1}, we define the \emph{universal cyclotomic Gaudin algebra}. It is a cyclotomic generalization of the Gaudin models with irregular singularities defined in \cite{FFT}. 

We go on to solve, by Bethe ansatz, the special case in which the Lax matrix has simple poles at the origin and arbitrarily many finite points, and a double pole at infinity.

\end{abstract}
\maketitle
\setcounter{tocdepth}{1}
\tableofcontents

\section{Introduction}
Pick a primitive $T$th root of unity $\omega\in \Cx$, for some non-negative integer $T$ and let $\Gamma := \la \omega \ra\subset \Cx$ denote a copy of the cyclic group $\Z/T\Z$. Let $\g$ be a finite-dimensional semisimple Lie algebra and $\sigma:\g\to\g$ an automorphism whose order divides $T$. 

Associated to these data is a \emph{cyclotomic Gaudin algebra} \cite{VY1}. It is a large commutative subalgebra of $U(\g^{\oplus N})^{\g^\sigma}$, depending on a choice of non-zero marked points $\zs = \{ z_1,\dots,z_N \}$ in the complex plane whose $\Gamma$-orbits are pairwise disjoint. 
It is generated by a hierarchy of Hamiltonians, among which are quadratic Hamiltonians $\mc H_1$, \dots, $\mc H_N$ that 
have appeared previously in \cite{Skrypnyk1,Skrypnyk2} -- see also \cite{CrampeYoung} -- and, in the context of cyclotomic KZ equations, in \cite{Brochier}. It defines a quantum integrable model generalizing the quantum Gaudin model \cite{Gaudin}, to which it reduces in the special case $T=1$.

The cyclotomic Gaudin algebra was constructed in \cite{VY1} using the technology of coinvariants/conformal blocks of $\gh$-modules of critical level, following \cite{FFR}. The relevant coinvariants in this case are $\Gamma$-equivariant; see \cite{VY2}.

Now, in fact, this approach using coinvariants naturally gives commutative subalgebras not just of $U(\g^{\oplus N})$ but of the larger algebra $U\left(\bigoplus_{i=1}^N \g[[t-z_i]]\right)$, where $\g[[t-z]]\cong \g[[t]]$ is the half loop algebra.
Moreover, in the cyclotomic setting it is natural to include also $0$ and $\8$ as marked points. These are the fixed points of the action of $\Gamma$, and to them one attaches twisted half loop algebras, respectively $\g[[t]]^\Gamma$ and $(t^{-1}\g^{\op}[[t^{-1}]])^\Gamma$ (see \S\ref{opsec}). The first main result of the present sequel to \cite{VY1} is thus to construct, in \S\ref{ZTGsec}, a large commutative subalgebra \be\nn\Gaud_{\zsa}(\g,\sigma)^\Gamma\subset U\left((t^{-1} \g^{\op}[[t^{-1}]])^\Gamma \oplus \bigoplus_{i=1}^N\g[[t-z_i]] \oplus (\g[[t]])^\Gamma\right)^{\g^\sigma}.\nn\ee 
It is the cyclotomic generalization of the \emph{universal Gaudin algebra} defined in \cite{FFT}. 

Quotients of half loop algebras of the form $\g[[t]] / t^n\g[[t]]$  are called \emph{(generalized) Takiff} algebras. Taking such quotients of $\Gaud_{\zsa}(\g,\sigma)^\Gamma$  one obtains commutative algebras $\Gaud_{\zsa}^\nsa(\g,\sigma)^\Gamma$.\footnote{Actually, we define these $\Gaud_{\zsa}^\nsa(\g,\sigma)^\Gamma$ first, and then the universal algebra $\Gaud_{\zsa}(\g,\sigma)^\Gamma$ is their inverse limit.} 
In particular, one recovers the cyclotomic Gaudin algebra of \cite{VY1} as the special case
\begin{align} \Gaud_{\zsa}^{1,(1),0}(\g,\sigma)^\Gamma \subset 
U(0\oplus \g^{\oplus N} \oplus 0)^{\g^\sigma}. \nn
\end{align}
More generally, following \cite{FFT} it is natural to call the integrable models defined by representing $\Gaud_{\zsa}^\nsa(\g,\sigma)^\Gamma$ on tensor products of modules over Takiff algebras, cyclotomic Gaudin models \emph{with irregular singularities}. 
Among the simplest possibilities is to introduce one irregular singularity, as mild as possible, at $\8$; that is, to consider
\begin{align} \Gaud_{\zsa}^{2,(1),1}(\g,\sigma)^\Gamma 
&\subset U\left( (\Pi_{-1} \g) \oplus \g^{\oplus N} \oplus \g^{\sigma}\right)^{\g^\sigma}.\nn
\end{align}
Here $\Pi_{-1}\g\subset \g$ denotes the $\omega^{-1}$-eigenspace of $\sigma$. It is to be regarded here as a commutative Lie algebra: it arises as the quotient $\Pi_{-1} \g \cong_\C (t^{-1} \g[[t^{-1}]])^{\Gamma}\big/ (t^{-2}\g[[t^{-1}]] )^\Gamma$. Suppose we now pick a one-dimensional representation of this commutative Lie algebra \ie a linear map $\chi: \Pi_{-1} \g \to \C$. Applying this map we obtain a commutative subalgebra 
 \be \Achi \subset \left(U(\g)^{\ox N}\ox U(\g^\sigma)\right)^{\g^\sigma_\chi},\nn\ee
where $\g^\sigma_\chi$ denotes the centralised of $\chi$ under the coadjoint action of $\g^\sigma$ on $(\Pi_{-1} \g)^\ast$. In the special case of $N=0$ (\ie only one marked point, at the origin) this is a cyclotomic generalization of the \emph{quantum shift-of-argument} subalgebra of \cite{Ryb06}; see also \cite{FFRb}. 

In the remainder of the paper we go on to diagonalize the Hamiltonians generating $\Achi$ on tensor products of Verma modules, by means of a Bethe ansatz. We assume that $\chi$, and the highest weights $\lambda_1,\dots,\lambda_N,\lambda_0$ of these Verma modules, all belong to the dual of a single Cartan subalgebra, and that this Cartan subalgebra is stable under $\sigma$. Under these assumptions one can apply the approach to the Bethe ansatz for Gaudin models from \cite{FFR,Freview}, which uses coinvariants of a particular class of $\gh$-modules at critical level called \emph{Wakimoto modules}. 
See \S\ref{sec: mr}, Theorem \ref{mthm}, for the precise statement of the result. 

Finally, in the special case $\chi=0$ we prove that the Bethe vectors are singular. See Theorem \ref{chi0thm}.

Let us conclude this introduction with some remarks and open questions.

As discussed in \cite{FFT}, the origin of the term \emph{irregular singularities} comes from the description of the spectrum of Gaudin algebras in terms of \emph{opers}. The notion of opers with regular singularities was recently extended to the cyclotomic setting in \cite{Lacroix:2016mpg}, and it was conjectured that the spectrum of the cyclotomic Gaudin algebra $\Gaud_{\zsa}^{1,(1),0}(\g,\sigma)^\Gamma$ admits a description in terms of such cyclotomic opers, or $\Gamma$-equivariant opers. It would be interesting to extend the definition of cyclotomic opers to include the case of irregular singularities and relate these to the spectrum of $\Gaud_{\zsa}^\nsa(\g,\sigma)^\Gamma$ defined in the present paper.

The quadratic Hamiltonians of the algebra $\Achi$ are the cyclotomic analogs of the Gaudin models considered in \cite{FMTV}, which exhibit a certain bispectrality property. It would be interesting to investigate bispectrality in the cyclotomic setting, in the spirit of that paper.

Cyclotomic analogs of the KP hierarchy were defined recently in \cite{CS}. This construction involves a generalization of (the completion of) Calogero-Moser phase space, which can be seen as a quiver variety whose underlying quiver has a single loop, to quiver varieties for cyclic quivers. Calogero-Moser space is known to be related to Gaudin algebras (Bethe algebras) \cite{MTVcalogero}, so it is natural to hope for a similar relation in the cyclotomic setting.

\section{The cyclotomic Gaudin model}\label{ZTGsec}
\subsection{Rational functions and formal series} We work over $\C$. For any formal variable $t$, we have the ring of polynomials $\C[t]$, the ring of formal power series $\C[[t]]$, and the field of formal Laurent series $\C((t))$. Given a finite collection of points $\xs = \{ x_1,\dots,x_p \} \subset \C$ let $\C_{\8,\xs}(t)$ denote the localization of $\C[t]$ by the multiplicative subset generated by $t-x_1,\dots,t-x_p$. Elements of $\C_{\8,\xs}(t)$ are rational functions in $t$ with poles at most at the points $x_1,\dots,x_p$ and at infinity. 

For any $z\in \C$ we have the map $\iota_{t-z} : \C_{\8,\xs}(t) \to \C((t-z))$ which returns the Laurent expansion $\iota_{t-z} f(t)$ of a rational function $f(t)$ about $t=z$. 
We have also $\iota_{t^{-1}}: \C_{\8,\xs}(t) \to \C((t^{-1}))$ which returns the Laurent expansion $\iota_{t^{-1}}f(t)$ of $f(t)$ in powers of $t^{-1}$. The maps $\iota_{t-z}$ and $\iota_{t^{-1}}$ are both injective homomorphisms of $\C$-algebras.

Let $\res_{t}: \C((t))  \to \C$ be the map which returns the coefficient of $t^{-1}$. 
For any $f(t) \in \C_{\8,\xs}(t)$ we have
\be  - \res_{t^{-1}} t^2 \iota_{t^{-1}} f(t)  + \sum_{i=1}^p \res_{t-x_i} \iota_{t-x_i} f(t)= 0.\label{zr}\ee
(This is equivalent to the statement that the sum of the residues of a meromorphic one-form $f(t) dt$ on $\CP^1$ vanishes.)

\subsection{Opposite Lie algebras and left vs. right modules}\label{opsec}
Given a complex Lie algebra $\mf a$ with Lie product $[\cdot,\cdot]: \mf a \ox \mf a \to \mf a$ we write $\mf a^\op$ for the opposite Lie algebra, namely the vector space $\mf a$ endowed with the Lie product $[X,Y]^\op := [Y,X]$. The Lie algebras $\mf a$ and $\mf a^\op$ are isomorphic (by e.g. $X\mapsto -X$) but it will be useful to regard them as two distinct Lie algebra structures on the same underlying vector space. Modules over $\mf a$ are naturally identified with left modules over the envelope $U(\mf a)$; modules over $\mf a^\op$ are naturally identified with right modules over the envelope $U(\mf a)$.

\subsection{Marked points and the group $\Gamma$}
Let $\omega$ be a root of unity of order $T\in \Z_{\geq 1}$. The cyclic group $\Gamma := \langle \omega \rangle \cong \Z/T\Z$ acts on the Riemann sphere $\CP^1= \C\cup \{\8\}$ by multiplication. 
The fixed points of this action are $0$ and $\8$. Pick $N\in \Z_{\geq 0}$ points $z_1,\dots,z_N \in \CP^1\setminus\{0,\8\}$ whose $\Gamma$-orbits are disjoint: $\Gamma z_i \cap \Gamma z_j = \emptyset$ whenever $i\neq j$. We write $\zs = \{z_1,\dots,z_N\}$.

Let $\g$ be a finite-dimensional simple Lie algebra over $\C$ and $\sigma: \g \to \g$ an automorphism of $\g$ whose order divides $T$.
Let $\la\cdot,\cdot\ra$ denote the invariant inner product on $\g$ normalised such that long roots have square length 2.
Let $\Pi_k$, $k\in \Z/T\Z$, be the projectors 
\be \Pi_k := \frac 1 T\sum_{m=0}^{T-1} \omega^{-mk} \sigma^m : \g \to \g\label{Pidef}\ee 
onto the eigenspaces of $\sigma$. They obey $\sum_{k\in \Z/T\Z} \Pi_k=\id$. We write $\g^\sigma$ for the subalgebra of invariants, 
\be \g^\sigma:= \Pi_0 \g.\nn\ee

Denote by $\g_\zsa^{\Gamma,k}$, $k\in \Z/T\Z$,  the Lie algebra of those $\g$-valued rational functions $f(t)$ of a formal variable $t$ that have no poles outside the set of points $\{0,\8\} \cup \Gamma\zs$ and that obey the equivariance condition $\omega^*f = \omega^k \sigma f$, \ie
\be f(\omega t) = \omega^k \sigma f(t) .\nn\ee
Let also  $(\g \ox \C((t^{\pm 1})))^{\Gamma,k} := \{ f(t^{\pm 1}) \in \g\ox \C((t^{\pm 1})): f(\omega^{\pm 1} t^{\pm 1}) = \omega^k \sigma f(t^{\pm 1})\}$. 
For brevity we write $\gG := \g_\zsa^{\Gamma,0}$ and $(\g \ox \C((t^{\pm 1})))^{\Gamma} := (\g \ox \C((t^{\pm 1})))^{\Gamma,0}$, etc. 

There is an injective homomorphism of $\C$-algebras 
\be  \g\ox \C_\zsa(t) \longrightarrow \g^\op \ox \C((t^{-1}))  \oplus \bigoplus_{i=1}^N \g\ox \C((t-z_i)) \oplus  \g\ox \C((t)) \nn\ee
defined by
\be f(t) \longmapsto ( - \iota_{t^{-1}} f(t); \iota_{t-z_1} f(t),\dots, \iota_{t-z_N} f(t); \iota_t(t)  ) \nn\ee
(note the $\op$ and minus sign in our conventions). 

\begin{lem}[$\Gamma$-equivariant Strong residue theorem]\label{lem:gsrt}
A tuple of formal series $$(f_\8; f_{z_1},\dots,f_{z_N}; f_0) \in (\g^\op \ox \C((t^{-1})))^{\Gamma,k}  \oplus \bigoplus_{i=1}^N \g \ox \C((t-z_i)) \oplus (\g \ox \C((t)))^{\Gamma,k}$$ belongs to $(-\iota_{t^{-1}}; \iota_{t-z_1},\dots, \iota_{t-z_N}; \iota_t)(\g_\zsa^{\Gamma,k})$, \ie they are the Laurent expansions of some rational function in $\g_\zsa^{\Gamma,k}$, if and only if
\be - \frac 1 T\res_{t^{-1}} t^2 \la f_\8 ,-\iota_{t^{-1}} g(t) \ra  +  \sum_{i=1}^N \res_{t-z_i} \la f_{z_i}, \iota_{t-z_i} g(t)\ra  + \frac 1T \res_{t} \la f_0, \iota_t(g)\ra 
 = 0\ee
for all $g(t) \in \g_\zsa^{\Gamma,-k-1}$.
\end{lem}
\begin{proof} The proof is as in \cite[Lemma A.1]{VY1}, but including the poles at $\8$. Compare  \eqref{zr}.\end{proof}

Let $\gh_{z_i}$ denote the extension of $\g\ox \C((t-z_i))$ by a one-dimensional centre $\C K_{z_i}$, defined by the cocycle
\be \Omega_{z_i}(f_{z_i}, g_{z_i}) := \res_{t-z_i} \la f_{z_i}, \del_{t-z_i} g_{z_i}\ra K_{z_i}, \qquad f_{z_i},g_{z_i} \in \g\ox \C((t-z_i)).\ee
Thus, each  $\gh_{z_i}$, $i=1,\dots, N$, is a copy of the (untwisted) affine Lie algebra $\gh$.

Let $\gh_0^\Gamma$ denote the extension of $(\g\ox \C((t)))^\Gamma$ by a one-dimensional centre $\C K_0$, defined by the cocycle 
\be \Omega_0(f_0,g_0) := \res_t \la f_0, \del_t g_0 \ra K_0,\qquad f_0,g_0 \in (\g\ox \C((t)))^\Gamma.\ee

Let $\gh_\8^{\op,\Gamma}$ denote the extension of $(\g^\op\ox \C((t^{-1})))^\Gamma$ by a one-dimensional centre $\C K_\8$, defined by the cocycle 
\begin{align} 
\Omega_\8(f_\8,g_\8) &:= \res_{t^{-1}} \la f_\8, \del_{t^{-1}} g_\8 \ra K_\8\nn\\ 
&= - \res_{t^{-1}} t^2 \la f_\8, \del_ tg_\8 \ra K_\8,\qquad f_\8,g_\8 \in (\g^\op\ox \C((t^{-1})))^\Gamma.
\end{align}

Given any $X \in \g$, $n \in \Z$ we introduce the notations
\begin{equation*}
X[n]_{z_i} := X \otimes (t-z_i)^n \in \gh_{z_i}, \qquad
X[n]_0 := X \otimes t^n \in \gh_0, \qquad
X[n]_\8 := X \otimes t^n \in \gh_\8.
\end{equation*}
Note in particular our conventions for the $n^{\rm th}$-modes at $\8$.

The algebras $\gh_0^\Gamma$ and $\gh_\8^\Gamma$ are both copies of an algebra $\gh^\Gamma$ which is either a twisted affine Lie algebra (if $\sigma$ is an outer automorphism) or else isomorphic to $\gh$ (if $\sigma$ is an inner automorphism).  

Let $\gN$ denote the extension of 
$(\g^\op \ox \C((t^{-1})))^{\Gamma} \oplus \bigoplus_{i=1}^N \g\ox\C((t-z_i)) \oplus (\g \ox \C((t)))^{\Gamma}$, by a one-dimensional centre $\C K$, defined by the cocycle 
\be\label{cocycle}
\Omega(f, g) :=  \left(- \frac 1 T \res_{t^{-1}} t^2 \la f_\8, \del_tg_\8 \ra 
+ \sum_{i = 1}^N \res_{t-z_i} \la f_{z_i},  \del_tg_{z_i} \ra 
 + \frac 1T \res_{t}  \la f_0,\del_tg_0\ra  \right)K,
\ee
where $f = (f_\8;f_{z_1},\dots,f_{z_N};f_0)$ and $g = (g_\8;g_{z_1},\dots,g_{z_N}; g_0)$ are in $(\g^\op \ox \C((t^{-1})))^{\Gamma}\oplus \bigoplus_{i=1}^N \g\ox\C((t-z_i)) \oplus (\g \ox \C((t)))^{\Gamma}$. 
In other words, $\gN$ is the quotient of the direct sum $\gh_\8^{\op,\Gamma} \oplus \bigoplus_{i=1}^N \gh_{z_i}\oplus \gh_0^\Gamma$ by the ideal spanned by $K_{z_i} - TK_0$, $i= 1, \ldots, N$, and $K_\8- K_0$, leaving one central generator, say $K_{z_1}$, which we call $K$. 

We have an embedding of Lie algebras
\be (-\iota_{t^{-1}}; \iota_{t-z_1},\dots,\iota_{t-z_N}; \iota_t ) : \gG \longhookrightarrow  
(\g^\op \ox \C((t^{-1})))^{\Gamma} \oplus \bigoplus_{i=1}^N \g\ox\C((t-z_i)) \oplus (\g \ox \C((t)))^{\Gamma}.\nn\ee
By Lemma \ref{lem:gsrt} the restriction of the cocycle $\Omega$ to the image of $\gG$ under this embedding vanishes. Therefore the embedding lifts to an embedding
\be \gG \longhookrightarrow \gN.\label{iotamap}\ee

\subsection{Induced $\gh_{\8, N, 0}$-modules}\label{sec: indg}
Let $\M_{z_i}$ be a module over $\g\ox \C[[t-z_i]]$, for each $i=1,\dots, N$. 
We then make it into a module over $\g\ox \C[[t-z_i]] \oplus \C K_{z_i}$ by declaring that $K_{z_i}$ acts by multiplication by $k\in \C$.  
Then we have the induced left $U(\gh_{z_i})$-module,
\be
\MM_{z_i}^k 
:=  U(\gh_{z_i})\ox_{U(\g\ox\C[[t-z_i]] \oplus \C K_{z_i})} \M_{z_i}.
\ee

Let $\M_0$ be a module over $(\g \ox \C[[t]])^\Gamma$. 
We make it into a module over $(\g\ox\C[[t]])^\Gamma \oplus \C K_0$ by declaring that  $K_0$ acts by multiplication by $k/T\in \C$. 
Then we have the induced left $U(\gh^\Gamma_{0})$-module,
\be
\MM_{0}^{k/T} 
:=  U(\gh^\Gamma_{0})\ox_{U((\g\ox\C[[t]])^\Gamma \oplus \C K_0)} \M_{0}.
\ee

Let $\M_\8$ be a module over $(\g^\op\ox t^{-1}\C[[t^{-1}]])^{\Gamma}$. We make it into a module  
over $(\g^\op\ox t^{-1}\C[[t^{-1}]])^{\Gamma} \oplus \C K_\8$ by declaring that $K_\8$ acts by multiplication by $k/T\in \C$. We have the induced left $U(\gh^{\op,\Gamma}_{\8})$-module,
\begin{align}
\MM_{\8}^{k/T} &:= 
U(\gh^{\op,\Gamma}_{\8}) \ox_{U((\g^\op\ox t^{-1}\C[[t^{-1}]])^\Gamma \oplus \C K_\8)} \M_{\8}
\end{align}
Let us write
\be \Mx := \M_\8 \ox \bigotimes_{i=1}^N \M_{z_i} \ox \M_0 \label{Mxdef}\ee
and 
\be \gNpp:=(\g^\op \ox t^{-1}\C[[t^{-1}]])^{\Gamma} \oplus \bigoplus_{i=1}^N \g\ox\C[[t-z_i]] \oplus (\g \ox \C[[t]])^{\Gamma} \oplus \C K.\label{gNpdef}\ee 
Then the tensor product
\be \MMx := \MM^{k/T}_\8 \ox \bigotimes_{i=1}^N \MM^k_i \ox \MM^{k/T}_0 =  U(\gN) \ox_{U(\gNpp\oplus \C K)} \Mx\label{MMxdef}\ee
is a module over $\gN$ on which $K$ acts as $k$. Pulling back by the embedding \eqref{iotamap}, we have that $\MMx$ becomes a module over $\gG$ and we can form the space of \emph{coinvariants},
\be \MMx \big/ \gG := \MMx \big/ (\gG\on \MMx) .\nn\ee
\begin{prop}\label{prop:isom1} The Lie algebras $\gNpp$ and $\gG$ embed as a pair of complementary Lie subalgebras in $\gN$, \ie
\be \gN = \gNpp \dotplus \gG\nn\ee
as vector spaces. Therefore there is a canonical isomorphism of vector spaces
\be \MMx \big/ \gG \cong_\C \Mx.\ee
\end{prop}
\begin{proof} As in \cite{VY1}, Lemma 2.1 and Corollary 2.4. \end{proof}

\subsection{Vacuum verma module $\VV_0^k$}
Now let $u\in \Cx$ be an additional non-zero marked point, whose orbit $\Gamma u$ is disjoint from $\Gamma \zs$. 
Then we have the algebras $\gGu$, $\gNu$, etc. defined as above but with the point $u$ included. To the point $u$ we assign a copy of the vacuum Verma module $\VV_0^k$ over the local copy $\gh_{u}$ of the affine Lie algebra $\gh$. Recall that by definition $\VV_0^k$ is the induced module
\be \VV_0^k  := U(\gh_{u})\ox_{U(\g\ox\C[[t-u]] \oplus \C K_{u})} \C \vac.\ee
Here $\C\vac$ denotes the one-dimensional module over  $\g \ox \C[[t-u]]\oplus \C K_{u}$ on which $\g\ox \C[[t-u]]$ acts trivially and $K_{u}$ acts by multiplication by $k\in \C$.  

A vector $X\in \VV_0^k$ is \emph{singular} if $A\on X=0$ for all $A\in \g \ox \C[t]$.
The singular vectors form a linear subspace of $\VV_0^k$ denoted $\ssc(\gh)$. 

\begin{prop}\label{prop:isom2} There is a canonical isomorphism of vector spaces
\[
\pushQED{\qed}
(\MMx\ox \VV_0^k) \big/ \gGu \cong_\C \Mx \ox \C\vac \cong_\C \Mx.
\qedhere
\popQED
\]
\end{prop}

It follows that, given any $X\in \VV_0^k$, there is a linear map $X(u):\Mx \to \Mx$ defined by
\be \Mx \longinto \MMx \xrightarrow{\cdot\ox X} \MMx \ox \VV_0^k \longonto (\MMx\ox \VV_0^k)\big/ \gGu \overset{\sim}\longrightarrow \Mx \label{Xumap}\ee
where $\Mx \into \MMx$ is the natural embedding. The map $X(u)$ depends rationally on $u$, with poles at most at the points $0$, $\omega^kz_i$ ($1\leq i\leq N$ and $k\in \ZT$) and $\8$. 

\subsection{Generalized Takiff algebras}
For any $n \in \Z_{\geq 1}$  there is an ideal $ \g \ox t^n \C[t] \subset  \g \ox \C[t]$. Define the Lie algebra $\tak n \g$ to be the quotient
\begin{align} \tak n\g &:= \left(\g\ox \C[t]\right) \big/ \left( \g \ox t^n \C[t] \right) \nn\\
&\cong_\C \g \oplus t\g \dots \oplus t^{n-1}\g.\nn\end{align}
Thus $\tak 1\g = \g$. The Lie algebra $\tak n\g$ is known as a \emph{(generalized) Takiff algebra}.   

The Lie algebras $\tak n \g$ together with the canonical projections $\tak n \g\onto \tak m\g$, $n>m$, form an inverse system, and $\g \ox \C[[t]]$ is the inverse limit $\varprojlim \tak n \g$. 

Define also the \emph{twisted Takiff algebra} $\tak n\g^\Gamma$:
\begin{align} \tak n\g^\Gamma &:= \left(\g\ox \C[t]\right)^\Gamma \big/ \left( \g \ox t^n \C[t] \right)^\Gamma \nn\\
&\cong_\C \g^\sigma \oplus t\Pi_1\g \dots \oplus t^{n-1}\Pi_{n-1}\g.\nn\end{align}
In particular $\tak 1\g^\Gamma = \g^\sigma$. 

We use the notation $X_{\ul p}$ for the class 1of the element $t^pX = X \ox t^p$ in $\tak n\g$.

For any $z\in \C$ we have the naive isomorphism $\tak n\g \cong \left(\g\ox \C[[t-z]]\right) \big/ \left( \g \ox (t-z)^n \C[[t-z]] \right)$ which  sends $X_{\ul p}$ to the class of $X\ox (t-z)^p$. 
By means of this isomorphism, modules over $\tak n \g$ pull back to modules over $\g\ox \C[[t-z]]$. 
 
\subsection{Universal Cyclotomic Gaudin Algebra}
Given any $n_0,n_{z_1},\dots,n_{z_N},n_\8\in \Z_{\geq 1}$ we write $\ns = \{ n_{z_1}, \ldots, n_{z_N} \}$. Let $\I_{\nsa}\subset \Upp$ denote the two-sided ideal in $\Upp$ generated by $\left( \g^\op \ox t^{-n_\8} \C[[t^{-1}]] \right)^\Gamma$, $\g \ox (t-z_i)^{n_{z_i}} \C[[t-z_i]]$, $i=1,\dots,N$, and $\left( \g \ox t^{n_0} \C[[t]] \right)^\Gamma$. 
Define 
\be \Upp_{\nsa} := \Upp/\I_{\nsa}.\label{Udef}\ee 
These 
form an inverse system whose inverse limit is $\Upp$.

Let us now take the module $\Mx$ in \eqref{Mxdef} to be a copy of $\Upp_\nsa$, regarded as a left module over itself. 
For any $X\in \VV_0^k$ we have a map $X(u) : \Upp_\nsa \to \Upp_\nsa$ as in \eqref{Xumap}. By construction this can be written in terms of the left action of $\Upp_\nsa$, which commutes with the right action of $\Upp_\nsa$. So $X(u)$ commutes with the right action of $\Upp_\nsa$. Hence for all $a\in \Upp_\nsa$, $X(u)\on a = X(u) \on (1a) = (X(u) \on 1) a$.  That is, $X(u)$ acts by left-multiplication by the element $X(u)\on 1\in \Upp_\nsa$. Since the latter depends on the choice of $\nsa$, we will denote it by $X(u)_\nsa$. When the choice of $\nsa$ is clear from the context we will write $X(u)_\nsa$ simply as $X(u)$. By construction, whenever $n_0'>n_0$, $n'_{z_i}> n_{z_i}$ and $n_\8'> n_\8$ then
\be X(u)_\nsa = X(u)_{n_0',\ns',n_\8'} + \I_{\nsa}. \label{invsys}\ee
In other words, the elements $X(u)_\nsa \in \Upp_\nsa$ are compatible with the above inverse system and hence define an element of the inverse limit $\Upp$. By a slight abuse of notation we will also call this element simply $X(u) \in \Upp$.

We have the natural inclusion $\g^\sigma \into \g$ and hence the ``diagonal'' embedding 
\begin{align}
\g^\sigma \longinto \g^{\sigma,\op} \oplus \bigoplus_{i=1}^N \g \oplus \g^\sigma&\longinto 
\gN,\nn\\  X\longmapsto (-X; X,\dots,X;X) &\longmapsto (-X[0]_\8; X[0]_{z_1},\dots, X[0]_{z_N};X[0]_0) .\label{ed}\end{align}
Identifying $\g^\sigma$ with its image under this embedding, this gives an action of $\g^\sigma$ on $U(\gN)$ by left- and right-multiplication. In particular, we can define the adjoint action of $\g^\sigma$ on $U(\gN)$. Note that the adjoint action stabilises $\Upp$ but the actions by left- and right-multiplication do not, because the zero-modes at $\8$ are not present in $\Upp$.
Let us write 
\be \Upp^{\gs} := \{ x\in \Upp: [a,x] =0\text{ for all } a\in \gs \}\ee
for the invariant subspace of the adjoint action of $\gs$ on $\Upp$. Define $\Upp_\nsa^\gs$ likewise.

Now suppose $X$ is a singular vector, $X\in \ssc(\gh) \subset \VV_0^k$.  Then $X$ is in particular $\g^\sigma$-invariant, $a\on X = 0$ for all $a\in \gs$. Hence we have
\be 0= [ 1 \ox a \on X ]  = -[ a\on 1 \ox X] = [X(u) a \on 1 \ox \vac ] = [X(u) a\on 1] \nn\ee
where we ``swapped using the constant rational function $a$'', \ie used $[a.(1\ox X)]=0$, in the second equality, used the definition of $X(u)$ in the third, and finally used  the isomorphism $(\MMx\ox \VV_0^k) \big/ \gGu \cong \MMx \big/ \gG$. On the other hand, in the space of coinvariants $\MMx \big/ \gG$ we have
\be 0= [a.(X(u)\on 1)] = [ aX(u) \on 1].\nn\ee
Taking the difference of the two equalities above, we get
\be 0= \big[ [a,X(u)] \on 1 \big] =  [a,X(u)] \nn\ee
where in the last equality we can use the identification $\MMx \big/ \gG \cong \Mx$. (The point is that neither $aX(u)$ nor $X(u)a$ need belong to $\Mx = \Upp_\nsa$, but the commutator $[a,X(u)]$ does, as we noted above.) 
This shows that if $X$ is singular then $X(u) \in \Upp_\nsa^\gs$.
 
For each $X\in \VV_0^k$, the element $X(u)$ depends rationally on $u$ with poles at most at $0, \omega^k z_i, \8$, $i=1,\dots,N$. 
Define the algebra $\Gaud_\zsa^\nsa(\g,\sigma)^\Gamma$ to be the span, in $\Upp_\nsa^\gs$, of all the coefficients of singular terms of Laurent expansions of the operators $Z(u)$ as $Z$ varies in the space of singular vectors $\ssc(\gh)\subset \Vcrit$. 

By virtue of \eqref{invsys}, the algebras $\Gaud_\zsa^\nsa(\g,\sigma)^\Gamma$ form an inverse system. Define the \emph{universal cyclotomic Gaudin algebra} $\Gaud_\zsa(\g,\sigma)^\Gamma$ to be the inverse limit,
\be \Gaud_\zsa(\g,\sigma)^\Gamma := \varprojlim \Gaud_\zsa^\nsa(\g,\sigma)^\Gamma.\nn\ee

By the argument  in \cite{VY1}, following \cite{FFR}, we have
\begin{thm}\label{Zthm}
Each $\Gaud_\zsa^\nsa(\g,\sigma)^\Gamma$ is a commutative subalgebra of $\Upp_\nsa^\gs$.\\
Hence $\Gaud_\zsa(\g,\sigma)^\Gamma$ is a commutative subalgebra of $\Upp^\gs$. \qed
\end{thm}


\subsection{Quadratic cyclotomic Hamiltonians}
Let $I_a\in \g$ and $I^a\in \g$, $a=1,\dots,\dim\g$, be dual bases of $\g$ with respect to $\la\cdot,\cdot\ra$, \ie $\la I_a,I^b\ra = \delta_a^b$. 
Let $\mc C:=\half I^{a}I_a \in Z(U(\g))$, the quadratic Casimir of $\g$. Here and below we employ summation convention on the index $a=1,\dots, \dim\g$. 
Define an element $F\in \g^\sigma$ and number $K \in \C$ by
\be F := \half \sum_{p=1}^{T-1} \frac{\omega^p [ \sigma^p I^a, I_a] }{\omega^p - 1}, \qquad 
K := \half \sum_{p=1}^{T-1} \frac{\omega^p \la \sigma^p I^a,I_a\ra k}{(\omega^p - 1)^2}.\label{FKdef}\ee

The \emph{quadratic Segal-Sugawara} vector $S$ is by definition
\be S := \half I^a[-1] I_a[-1] \vac \in \VV_0^k \label{Sdef}\ee
At the critical level $k=-h^\vee$, the vector $S$ is singular.  
\begin{prop}\label{qHprop}
The corresponding element $S(u) \in \Upp^\gs$ is given by
\begin{align} S(u) - \frac1 {u^2}K 
&= \sum_{i=1}^N \sum_{k=0}^{T-1} \sum_{p=0}^\8 \frac{\omega^{-kp+k} \mc H_{i,p}}{(u-\omega^{-k}z_i)^{p+1}} 
+ \sum_{\substack{p=0\\ p \equiv 1 \, \textup{mod}\, T}}^\8 \frac{\mc H_{0,p}}{u^{p+1}} 
+ \sum_{\substack{p=0\\ p \equiv -2 \, \textup{mod}\, T}}^\8 u^p \mc H_{\8,p} \end{align}
where 
\begin{multline} 
\mc H_{i,p} =  \sum_{\substack{j=1\\j\neq i}}^N \sum_{l=0}^{T-1} \sum_{n,m = 0}^\8  \frac{(-1)^n \binom{n+m} m}{(z_i - \omega^{-l} z_j)^{n+m+1}}  I_a[n+p]_{z_i} \omega^{-lm}  (\sigma^{l} I^a)[m]_{z_j} \\
+\sum_{l=1}^{T-1} \sum_{r,m = 0}^\8  \frac{\omega^{-lm} (-1)^r \binom{r+m} m}{((1  - \omega^{-l}) z_i)^{r+m+1}} 
\half \left\{I_a[r+p]_{z_i}, (\sigma^{l} I^a)[m]_{z_i}\right\} + \sum_{n=0}^{p-1} \half I_a[n]_{z_i} I^a[p-n-1]_{z_i}\\
+ T \sum_{n,m=0}^\8   \frac {(-1)^{n} \binom {n+m} m}{z_{i}^{n+m+1}} I_a[n+p]_{z_i} (\Pi_m I^a)[m]_0 + \sum_{n=0}^\8   \frac{(-1)^{n}}{z_i^{n+1}}  F[n+p]_{z_i}\\
+ T \sum_{n,m=0}^\8  z_i^{n} \binom {n+m} m  (\Pi_{-n-m-1} I_a)[-n-m-1]_\8  I^a[p+m]_{z_i},
\end{multline} 
for $i = 1, \ldots, N$, and
\begin{multline}
\mc H_{0,p} = T^2 \sum_{m=0}^\8\sum_{n=0}^\8 \sum_{i=1}^N \frac{(-1)^{n+p+1} \binom {n+m} m }{z_{i}^{n+m+1}} (\Pi_{p-m-1}I_a)[n+p]_{z_i} (\Pi_m I^a)[m]_0\\
+ T^2 \sum_{n=0}^p (\Pi_n I_a)[n]_0 (\Pi_{-p+n-1} I^a)[-p+n-1]_\8  \\
+\frac{T^2}{2} \sum_{n=0}^{p-1} (\Pi_n I_a)[n]_0 (\Pi_{p-n-1} I^a)[p-n-1]_0 + T (\Pi_{p-1} F)[p-1]_0,
\end{multline}
and where finally
\begin{multline}
\mc H_{\8,p} = T^2 \sum_{i=1}^N \sum_{m,n=0}^\8 z_i^{m-n-1-p} \binom {m-1-p} n (\Pi_{-m-1} I_a)[-m-1]_\8 (\Pi_{m-p-1} I^a)[n]_{z_i} \\
+ T^2 \sum_{n=0}^\8 (\Pi_n I_a)[n]_0 (\Pi_{-p-n-2} I^a)[-p-n-2]_\8 \\
+ \frac{T^2}{2} \sum_{n=0}^p (\Pi_{-n-1} I_a)[-n-1]_\8 (\Pi_{-p+n-1} I^a)[-p+n-1]_\8 + T (\Pi_{-p-2} F)[-p-2]_\8.
\end{multline}
\end{prop}
\begin{proof}
The proof is given in Appendix \ref{qHapp}.
\end{proof}

\subsection{Regular singularities and shift-of-argument/twisted boundary conditions}
In the special case when $n_{z_i}=1$ for $i=1,\dots, N$ and $n_0 = 1$, we obtain commutative subalgebras 
\be \Gaud^{n_\8,(1),1}_{\8,\zs,0}(\g,\sigma)^\Gamma \subset 
\left(U\Big(\big(\g^\op\ox t^{-1} \C[t^{-1}]\big)^\Gamma \big/ \big( \g^\op \ox t^{-n_\8} \C[t^{-1}] \big)^\Gamma\Big) \ox U(\g)^{\ox N}\ox U(\g^\sigma)\right)^{\g^\sigma}.\nn\ee

If furthermore we set $n_\8=1$ then we obtain the commutative subalgebra
\be \Gaud^{1,(1),1}_{\8,\zs,0}(\g,\sigma)^\Gamma \subset (U(\g)^{\ox N}\ox U(\g^\sigma))^{\g^\sigma}. \nn\ee

Now consider setting $n_\8=2$. 
Pick any linear map 
\be \chi:\Pi_{-1}\g \to \C.\nn\ee 
The Lie algebra $\left(\g^\op\ox t^{-1} \C[t^{-1}]\right)^\Gamma \big/ \left( \g^\op \ox t^{-2} \C[t^{-1}] \right)^\Gamma$ is commutative, and is canonically isomorphic to $\Pi_{-1}\g$ as a vector space. We may therefore regard $\chi$ as an algebra homomorphism $\chi: U\big(\left(\g^\op\ox t^{-1} \C[t^{-1}]\right)^\Gamma \big/ \left( \g^\op \ox t^{-2} \C[t^{-1}] \right)^\Gamma\big)\to \C$. 
Let then 
\be \Achi := (\chi\ox \id^{\otimes N} \ox \id)\left(\Gaud^{2,(1),1}_{\8,\zs,0}(\g,\sigma)^\Gamma\right) .\nn\ee
This defines a commutative subalgebra 
\be \Achi \subset \left(U(\g)^{\ox N}\ox U(\g^\sigma)\right)^{\g^\sigma_\chi},\nn\ee
where $\g^\sigma_\chi = \left\{ X\in \g^\sigma: \chi([X,Y])=0 \text{ for all } Y\in \Pi_{-1} \g\right\}$ denotes the centralizer of $\chi$ under the coadjoint action.
Note that in the special case $\chi = 0$ we recover $\Gaud^{1,(1),1}_{\8,\zs,0}(\g,\sigma)^\Gamma$.

For any $X \in \g$ and $i = 1, \ldots, N$ we let $X^{(i)}$ denote the element of $U(\g)^{\otimes N} \otimes U(\g^\sigma)$ with $X$ in the $i^{\rm th}$ tensor factor and a $1$ everywhere else. Similarly, for $X \in \g^\sigma$ we let $X^{(0)}$ be the element $1^{\otimes N} \otimes X$ of $U(\g)^{\otimes N} \otimes U(\g^\sigma)$. The only non-zero Hamiltonians of Proposition \ref{qHprop} above are then
\be
\mc H_{i,0} =  \sum_{\substack{j=1\\j\neq i}}^N \sum_{l=0}^{T-1}  \frac{I_a^{(i)}  (\sigma^{l} I^a)^{(j)}}{(z_i - \omega^{-l} z_j)} +\sum_{l=1}^{T-1}  \frac{(\sigma^l I_a)^{(i)} I^{a(i)}}{(1  - \omega^{-l}) z_i} 
+ T \frac {1}{z_{i}} I_a^{(i)} (\Pi_0 I^a)^{(0)} + T I^{a(i)} \chi(\Pi_{-1} I_a),
\nn\ee
\be \mc H_{i,1} = \half I_a^{(i)} I^{a(i)},\nn\ee
\be
\mc H_{0,0} = T^2 \sum_{i=1}^N \frac{(-1)}{z_{i}} (\Pi_{-1}I_a)^{(i)} (\Pi_0 I^a)^{(0)} + T^2 (\Pi_0 I_a)^{(0)} \chi(\Pi_{-1} I^a),
\nn\ee
\be
\mc H_{0,1} = \frac{T^2}{2} (\Pi_0 I_a)^{(0)} (\Pi_{0} I^a)^{(0)} + T F^{(0)},
\nn\ee
and 
\be
\mc H_{\8,0} = \frac{T^2}{2} \chi(\Pi_{-1} I_a) \chi(\Pi_{-1} I^a).
\nn\ee
Note that $\mc H_{0,0} = 0$ unless $T=1$ and $\mc H_{\8,0} = 0$ unless $T=1$ or $2$.

\begin{rem} The cyclotomic Gaudin algebra introduced in \cite{VY1} is the commutative subalgebra $\Gaud^{1,(1),0}_{\8,\zs,0}(\g,\sigma)^\Gamma \subset (U(\g)^{\ox N})^{\g^\sigma}$. It is the image of $\Gaud^{1,(1),1}_{\8,\zs,0}(\g,\sigma)^\Gamma$ under $\id^{\ox N} \ox \eps$, where $\eps: U(\g^\sigma) \to \C$ is the counit. 
\end{rem}

\begin{rem} The algebra $\Achi$ is a cyclotomic generalisation of the \emph{quantum shift-of-argument subalgebra}; see \cite{Ryb06,FFT,FFRb}. 
\end{rem}
\begin{rem}
Sometimes setting $\chi\neq 0$ is called adding \emph{twisted boundary conditions}. The name comes from the Heisenberg XXX spin chain of which the usual Gaudin model is a limit.  
\end{rem}

\section{Statement of main results}\label{sec: mr}  
\subsection{Cartan data and Verma modules}\label{sec: cd}
We fix a Cartan decomposition $\g = \n_- \oplus \h \oplus \n$ of $\g$. Let $\Delta^+\subset \h^*$ be the set of positive roots of $\g$ and $\{\alpha_i\}_{i\in I} \subset \Delta^+$ the set of simple roots, where $i$ runs over the set $I$ of nodes of the Dynkin diagram of $\g$. Let $E_{\alpha}$ (resp. $F_\alpha$) be a root vector of weight $\alpha$ (resp. $-\alpha$) for each root $\alpha\in \Delta^+$, and $H_\alpha \equiv \alpha^\vee := [E_\alpha,F_\alpha]$ the corresponding coroot. Overloading notation somewhat, we write $H_i:= H_{\alpha_i}$, $i\in I$. Then $\{H_i\}_{i\in I}\cup \{E_\alpha,F_\alpha\}_{\alpha\in \Delta^+}$ is a Cartan-Weyl basis of $\g$. 
We shall assume the Cartan decomposition has been chosen to be compatible with the automorphism $\sigma:\g\to \g$, in the sense that $\sigma(\h) = \h$, $\sigma(\n) = \n$ and $\sigma(\n^-) = \n^-$. 
(Such a choice is always possible \cite{KacBook}.)

Let $M_{\lambda}$ denote the \emph{Verma module} over $\g$ with highest weight $\lambda\in \h^*$, namely 
\be M_\lambda := U(\g) \ox_{U(\h\oplus \n)} \C v_\lambda,\ee 
where $\C v_{\lambda}$ is the one-dimensional module over $\h\oplus \n$ generated by a vector $v_{\lambda}$ with $\n \on v_{\lambda}=0$ and $h\on v_{\lambda} =  \lambda(h)v_\lambda$ for all $h\in \h$. 
Similarly, let $M_\lambda^\sigma$ denote the Verma module over $\g^\sigma$ with highest weight $\lambda\in \h^{*,\sigma}$,  
\be M_\lambda^\sigma := U(\g^\sigma) \ox_{U(\h^\sigma\oplus \n^\sigma)} \C v_\lambda,\ee

\subsection{The weight function}
Let $\lambda_1,\dots,\lambda_N\in \h^*$ be $\g$-weights.
Let $\lambda_0\in \h^{*,\sigma}$ be a $\g^\sigma$-weight.

As above, let $z_1,\dots,z_N$ be a collection of $N\in \Z_{\geq 0}$ non-zero points in $\C$ such that $\Gamma z_i \cap \Gamma z_j = \emptyset$ for all $1\leq i< j\leq N$. 
In addition, let $w_1,\dots,w_m$ be a collection of $m\in \Z_{\geq 0}$ non-zero points in $\C$ such that $\Gamma w_i \cap \Gamma w_j = \emptyset$ for all $1\leq i<j\leq m$ and such that $\Gamma w_i\cap \Gamma z_j = \emptyset$ for all $1\leq i\leq m$ and $1\leq j\leq N$. 
Let $c(1),\dots, c(m)$ be 
elements of $I$. We call $c(i)$ the \emph{colour} of the variable $w_i$. 

Recall the projectors $\Pi_k$ from \eqref{Pidef}, and in particular $\Pi_0$. 
The \emph{(cyclotomic) weight function} $\psi$ associated to these data is the element
\be \psi \in \left(\bigotimes_{i=1}^N M_{\lambda_i}\ox M^\sigma_{\lambda_0}\right)_{\lambda_\8}, \quad\text{where}\quad
\lambda_\8 := \lambda_0 +\sum_{i=1}^N\Pi_0 \lambda_i  - \sum_{j=1}^m \Pi_0 \alpha_{c(j)} \label{l8def},\ee
defined recursively as follows. Define linear maps
$$\theta_s: \bigotimes_{i=1}^N M_{\lambda_i}\ox M^\sigma_{\lambda_0} \ox \underbrace{\n_- \ox\dots\ox \n_-}_s \longrightarrow \bigotimes_{i=1}^N M_{\lambda_i}\ox M^\sigma_{\lambda_0} \ox \underbrace{\n_- \ox\dots\ox \n_-}_{s-1} $$ for $s=1,2,\dots,m$, by
\begin{multline} \label{swappingtau} \theta_s(x_1\ox \dots\ox x_N\ox x_0 \ox y_1\ox\dots\ox y_s)\\
 =\frac{x_1\ox\dots\ox\dots \ox x_N\ox(T\Pi_0  y_s)x_0\ox y_1\ox \dots\ox  y_{s-1})}{w_{s}} \\\nn
+ \sum_{i=1}^N \sum_{j\in \Z_T} \frac{ x_1\ox\dots\ox x_{i-1} \ox (\sigma^jy_s) x_i\ox x_{i+1}\ox\dots \ox x_N\ox x_0\ox y_1\ox \dots\ox  y_{s-1})}{w_{s} - \omega^{-j}z_{i}} \\\nn
+  \sum_{i=1}^{s-1}\sum_{j\in \Z_T} \frac{x_1\ox \dots\ox  x_N\ox x_0\ox y_1\ox \dots\ox  y_{i-1}\ox [\sigma^{j}y_s,  y_i] \ox y_{i+1}\ox  \dots\ox  y_{s-1}}{w_{s} - \omega^{-j}w_{i}}\end{multline}
Then the weight function is by definition the element
\be
\psi := (-1)^m (\theta_1\circ \dots \circ \theta_m)(v_{\lambda_1}\ox \dots\ox v_{\lambda_N}\ox v_{\lambda_0}\ox  F_{\alpha_{c(1)}}\ox F_{\alpha_{c(2)}}\ox \dots\ox F_{\alpha_{c(m)}}).\label{psidef}\ee

\subsection{The weight $\Lambda_0$}
Define a weight $\Lambda_0\in \h^{*,\sigma}$ by 
\be \label{lambda0def}
\Lambda_0(h) := \sum_{r=1}^{T-1} \frac {\tr_\n (\sigma^{-r} \ad_h)} {1 - \omega^r} 
\ee
where $\ad_h:\n\to\n; X\mapsto [h,X]$ is the adjoint action of $\h$ on $\n$.  

For a more explicit expression for $\Lambda_0$, note that \cite[\S8.6]{KacBook}
\be \sigma(E_{\alpha}) = \tau_\alpha E_{\sigma(\alpha)},\qquad \sigma(H_i) = H_{\sigma(i)},\qquad \sigma(F_{\alpha}) = \tau_{\alpha}^{-1} F_{\sigma(\alpha)}.\label{sigmaE}\ee
Here, by overloading notation, we write $\sigma:\Delta^+\to\Delta^+$ for the symmetry of the root system, coming in turn from a symmetry $\sigma:I\to I$ of the Dynkin diagram. The numbers $\tau_\alpha$, $\alpha\in \Delta^+$, are certain roots of unity in $\Gamma= \omega^\Z$. (So the ``inner part'' of the automorphism $\sigma: \g\to \g$ is encoded in the choice of $\tau_{\alpha_i}$, $i\in I$; the remaining $\tau_\alpha$ are fixed by this choice.)

Then 
\be \Lambda_0 = \sum_{r=1}^{T-1} \frac{1}{1 - \omega^r} \sum_{\substack{\alpha\in \Delta^+\\\sigma^r(\alpha)=\alpha}} \left( \prod_{p=0}^{r-1} \tau_{\sigma^p(\alpha)}^{-1} \right) \alpha.\label{l0def}\ee

\subsection{The cyclotomic Bethe equations}
Given a complex vector space $A$, for any linear map $\eta : A \to \C$ we define $\lsigma \eta := \eta \circ \sigma^{-1}$. Let $\chi\in \h^*$ be such that $\lsigma \chi = \omega \chi$.
The \emph{cyclotomic Bethe equations (with twisted boundaries)} are
\begin{multline} \label{tbe}
0= \sum_{r=0}^{T-1} \sum_{i=1}^N\frac{\langle \alpha_{c(j)},\lsigma^r\lambda_i\rangle}{w_j-\omega^rz_i} - \sum_{r=0}^{T-1} \sum_{\substack{k=1\\k\neq j}}^m \frac{\langle \alpha_{c(j)},\lsigma^r\alpha_{c(k)}\rangle}{w_j-\omega^rw_k} \\+
\frac{1}{w_j}
\left(- \frac{1}{2} \sum_{r=1}^{T-1} \langle \alpha_{c(j)},\lsigma^r\alpha_{c(j)}\rangle + \langle \alpha_{c(j)}, \lambda_0+\Lambda_0 \rangle \right) + \langle \alpha_{c(j)}, \chi \rangle.
\end{multline}

\subsection{Eigenvectors of the cyclotomic Gaudin algebra} \label{sec: BAE}

Let $\chi\in \h^*$ be such that $\lsigma \chi = \omega \chi$. Extend $\chi$ to an element of $\g^*$ by setting $\chi(\n) = \chi(\n_-) = 0$. Then a $\chi$ is a linear map $\Pi_{-1} \g \to \C$. 

Suppose $(w_1,\dots,w_m;c(1),\dots,c(m))$ are such that cyclotomic Bethe equations \eqref{tbe} are satisfied.

\begin{thm} \label{mthm}
The weight function is an eigenvector of the algebra $\Achi$. In particular, 
the eigenvalues $E_i$ of the quadratic cyclotomic Gaudin Hamiltonians $\mathcal H_{i,0}$ are given by 
\be 
E_i := \sum_{\substack{j=1\\j\neq i}}^N \sum_{s=0}^{T-1} \frac{\langle \lambda_i,\lsigma^s \lambda_j\rangle}{z_i-\omega^sz_j} 
- \sum_{j=1}^m \sum_{s=0}^{T-1} \frac{\langle \lambda_i,\lsigma^s \alpha_{c(j)} \rangle}{z_i-\omega^sw_j}
+ \frac{1}{z_i} \left( \langle \lambda_i, \lambda_0+\Lambda_0\rangle
+ \frac{1}{2} \sum_{s=1}^{T-1} \langle \lambda_i,\lsigma^s \lambda_i\rangle \right) + \langle \lambda_i,\chi\rangle.
\nn\ee
\end{thm}
In particular, when $\chi=0$ the weight function is an eigenvector of the algebra $\Gaud^{1,(1),1}_{\8,\zs,0}(\g,\sigma)^\Gamma$. Moreover, in that case we have the following.
\begin{thm}\label{chi0thm}
In the special case $\chi=0$, the weight function $\psi$ belongs to the space of singular vectors $\left( \bigotimes_{i=1}^N M_{\lambda_i} \ox M_{\lambda_0}^{\sigma} \right)^{\n^\sigma}_{\lambda_\8}$. 
\end{thm}

\begin{rem} We have not shown that the weight function is non-zero. When $\sigma$ is a diagram automorphism, this is proved in \cite{VY16}.
\end{rem}

\section{Proofs}
\subsection{Restricted duals and contragredient Verma modules}
Given a module $M$ over $\g$ we write $(M)_{\mu}$ for the subspace of weight $\mu \in \h^*$,
\be (M)_\mu := \{ v\in M : \text{ there exists $n\in \Z_{\geq 1}$ such that } (h - \mu(h) 1)^n \on v = 0 \text{ for all $h\in \h$}\}. \ee 
The module $M$ is a \emph{weight module} if $M = \bigoplus_{\mu\in \h^*} (M)_\mu$. 
In this paper we work with weight modules all of whose weight subspaces are of finite dimension. If $M$, $N$ are two such modules, then by $\Hom_\C(M,N)$ we shall always mean the restricted space of linear maps 
\be \Hom_\C(M,N) := \bigoplus_{\mu,\nu\in \h^*} \Hom_\C\left((M)_\mu,(N)_\nu\right). \nn\ee
In particular we shall write $M^*:=\Hom_\C(M,\C) = \bigoplus_{\mu\in \h^*} ((M)_\mu)^*$, \ie our duals are restricted duals. We have $\Hom_\C(M,N) = \Hom_\C(M\ox N^*,\C) = \Hom_\C(\C,M^*\ox N)= M^* \ox N$. 

The restricted dual $M^*_\lambda$ of the Verma module $M_\lambda$ is naturally a right $U(\g)$-module. We may twist by any anti-automorphism of $U(\g)$ to obtain a left module. The Cartan anti-automorphism $\cai: U(\g) \to U(\g)$ is defined by
\be \cai (H_i) = H_i, \quad i\in I,\quad\text{and}\quad \cai(E_\alpha) = F_\alpha,\quad \cai(F_\alpha) = E_\alpha, \quad \alpha\in \Delta^+.\ee
It obeys $\cai^2=\id$. The twist of $M^*_\lambda$ by $\cai$ is the left $U(\g)$-module called the \emph{contragredient Verma module}. Henceforth by $M^*_\lambda$ we shall always mean the restricted dual equipped with this left $U(\g)$-module structure. That is
\be (x\on f)(v) := f(\cai(x) \on v), \qquad f\in M^*_\lambda, \quad x\in \g, \quad v\in M_\lambda.\nn\ee
See e.g. \cite[\S3.3]{HumphreysBGG}. 

Let $S:U(\g)\to U(\g)$ be the antipode map, \ie the extension of the map $\g\to \g;X\to -X$ to an anti-automorphism of $U(\g)$. We have the automorphism $\cai \circ S = S\circ \cai$ of $\g$. Let us write $(M_\lambda^*)^{\cai\circ S}$ for the left $U(\g)$-module obtained by twisting by this automorphism. In other words $(M_\lambda^*)^{\cai\circ S}$ is the dual of $M_\lambda$ in the usual Hopf-algebraic sense.  
Hence we have 
\be \Hom_\g(A,B\ox M_\lambda) = \Hom_\g(A\ox (M_\lambda^*)^{\cai\circ S}, B).\label{sho}\ee

Similarly, one has the notion of weight modules, contragredient Verma modules, etc., over $\g^\sigma$.

Given a module $V$ over a Lie algebra $\mf a$, we denote by $V^{\mf a}$ the space of invariants $V^{\mf a} := \{ x\in V: a\on x = 0\,\text{ for all } a\in \mf a\}$.

\subsection{Heisenberg algebras at the marked points}
Let $z_1,\dots,z_N$, $w_1,\dots,w_m$, be as in \S\ref{sec: mr}. For brevity we introduce $p:= N+m$ and $(x_1,\dots,x_p) := (z_1,\dots,z_N,w_1,\dots,w_m)$. Let $\xs := \{ x_1, \ldots, x_p \}$.

Let $\n_\C$ (resp. $\n^*_\C$) denote the vector space $\n$ (resp. $\n^*$) endowed with the structure of a commutative Lie algebra. On the commutative Lie algebra $\n_\C\oplus\n^*_\C$ there is a non-degenerate bilinear skew-symmetric form $\langle\cdot,\cdot\rangle$ defined by
\be \langle X, Y \rangle = Y|_{\n^*}\left(X|_{\n}\right) - X|_{\n^*}\left(Y|_{\n}\right),  \ee
and an action by automorphisms of the group $\Gamma$ given by 
\be \omega. X := \sigma(X|_{\n}) \oplus \lsigma(X|_{\n^*}).\label{Gn}\ee

Let $\Heis_{x_i}$, $i=1,\dots,p$, denote the central extension of the commutative Lie algebra $(\n_\C\oplus \n^*_\C)\ox \C((t-x_i))$, by a one-dimensional centre $\C\mathbf 1_{x_i}$, defined by the cocycle $\res_{t-x_i} \langle f , g \rangle\mathbf 1_{x_i}$.
Let $\Heis_{0}^\Gamma$ denote the extension of $(\n_\C \ox \C((t)))^{\Gamma,0} \oplus (\n^*_\C\ox \C((t)))^{\Gamma,-1}$ by a one-dimensional centre $\C \mathbf 1_{0}$ defined by the cocycle $\res_{t} \langle f , g \rangle\mathbf 1_{0}$. 
Let $\Heis_{\8}^\Gamma$ denote the extension of $(\n_\C \ox \C((t^{-1})))^{\Gamma,0} \oplus (\n^*_\C\ox \C((t^{-1})))^{\Gamma,-1}$ by a one-dimensional centre $\C \mathbf 1_{\8}$ defined by the cocycle $\res_{t^{-1}} t^2 \langle f , g \rangle\mathbf 1_{\8}$. 

Let us give a more explicit description of these Lie algebras in terms of generators and relations. 
To do so, we first construct bases of $\n$ and $\n^*$ adapted to the automorphism $\sigma$. 
Recall the projectors $\Pi_k$, $k\in \Z/T\Z$, from \eqref{Pidef}. By the adjoint action, $\g$ is a module over itself. In particular, it is a module over its Lie subalgebra $\g^\sigma = \Pi_0 \g$. As a $\g^\sigma$-module, $\g = \bigoplus_{k\in \Z/T\Z} \Pi_k \g$. Let $\Delta_k^+$ denote the set of $\g^\sigma$-weights of $\Pi_k \n$ and for $\alpha\in \Delta_k^+$ let $\n_{(k,\alpha)}$ denote the corresponding weight subspace of $\Pi_k \n$.    
We may pick a basis of $\n$ consisting of vectors $E_{(k,\alpha)}\in \n_{(k,\alpha)}$, where $k\in \Z/T\Z, \alpha\in \Delta^+_k$.\footnote{Indeed, suppose $\alpha\in \Delta^+$ is a root of $\g$ such that the orbit $\sigma^\Z \alpha$ has $t$ elements, where $t\in \Z_{\geq 1}$ divides $T$. Then $\sigma^t E_\alpha = \omega^{tm}E_\alpha$ for some unique $m\in \Z/T\Z$. Let $E_{(m-kT/t,\bar\alpha)} := \sum_{j=0}^{t-1} \omega^{-(m-kT/t)j} \sigma^j E_{\alpha}\in \n_{(m-kT/t,\bar\alpha)}$ for $k=0,1,\dots,t-1$. By picking one root from each $\sigma$-orbit, we construct a basis of the required form.} 

We now have two bases of $\n$, namely $E_{\alpha}$, $\alpha\in \Delta^+$, and $E_{(k,\alpha)}$, $k\in \Z/T\Z$, $\alpha\in \Delta_k^+$. We write $E_{\alpha}^*$, $\alpha\in\Delta^+$, and $E_{(k,\alpha)}^*$, $k\in \Z/T\Z$, $\alpha\in \Delta_k^+$  for their respective dual bases of $\n^*$.

Then $\Heis_{x_i}$ has the following explicit set of generators:
\be a_\alpha[n]_{x_i} := E_{\alpha}\ox (t-x_i)^n, \quad a^*_\alpha[n]_{x_i} := E_\alpha^*\ox (t-x_i)^{n-1}, 
\nn\ee 
where $\alpha\in \Delta^+$ and $n\in \Z$; while explicit sets of generators for $\Heis_{0}^\Gamma$ and $\Heis_{\8}^\Gamma$ are
\begin{alignat}{3} 
a_{(k,\alpha)}[nT+k]_{0} &:= E_{(k,\alpha)}  \ox t^{nT+k} & &\in (\n_\C \ox \C((t)))^{\Gamma,0}, \nn\\    
a^*_{(k,\alpha)}[nT+k]_{0} &:= E_{(k,\alpha)}^*\ox t^{nT+k-1}& &\in (\n_\C^*\ox \C((t)))^{\Gamma,-1} ,\nn
\end{alignat}
and
\begin{alignat}{3} 
a_{(k,\alpha)}[nT+k]_{\8} &:= E_{(k,\alpha)}  \ox t^{nT+k} & &\in (\n_\C \ox \C((t^{-1})))^{\Gamma,0}, \nn\\    
a^*_{(k,\alpha)}[nT+k]_{\8} &:= E_{(k,\alpha)}^*\ox t^{nT+k-1}& &\in (\n_\C^*\ox \C((t^{-1})))^{\Gamma,-1}, \nn
\end{alignat} 
respectively, where $k\in \Z/T\Z$, $\alpha\in \Delta^+_k$, and $n\in \Z$. 

\begin{rem} Note our conventions for the modes at $\8$.  \end{rem}

These generators obey the relations:
\begin{align} 
[a_\alpha[n]_{x_i},a^*_\beta[m]_{x_j}] &= \delta_{ij} \delta_{\alpha\beta} \delta_{n,-m} \mathbf 1_{x_i},\nn\\
[a_{(i,\alpha)}[n]_{0}, a^*_{(j,\beta)}[m]_{0} ]  &= \delta_{ij} \delta_{\alpha\beta} \delta_{n,-m} \frac 1 T \mathbf 1_{0},  \nn\\
[a_{(i,\alpha)}[n]_{\8}, a^*_{(j,\beta)}[m]_{\8} ]  &= \delta_{ij} \delta_{\alpha\beta} \delta_{n,-m} \frac 1 T \mathbf 1_{\8},  
\nn\end{align}
with all other commutators vanishing. 

Each $\Heis_{x_i}$ is isomorphic to the \emph{Heisenberg Lie algebra} $\Heis(\g)$, while $\Heis_{0}^\Gamma$ and $\Heis_{\8}^\Gamma$ are isomorphic to a subalgebra $\Heis(\g)^\Gamma$. Note  that the opposite Lie algebra $\Heis_{\8}^{\Gamma,\op}$ differs from $\Heis_{\8}^\Gamma$ only in the sign of the central extension. 

Let $\Heis_{\8, p, 0}$ denote the extension of 
\begin{multline} (\n_\C \ox \C((t^{-1})))^{\Gamma,0} \oplus (\n^*_\C\ox \C((t^{-1})))^{\Gamma,-1}\\
{}\oplus 
\bigoplus_{i=1}^p (\n_\C\oplus \n^*_\C)\ox \C((t-x_i))\\
\oplus
(\n_\C \ox \C((t)))^{\Gamma,0} \oplus (\n^*_\C\ox \C((t)))^{\Gamma,-1} 
\label{dH}\end{multline}
by a one-dimensional centre $\C \mathbf 1$, defined by the cocycle 
\be\label{cocycle}
\Omega(f, g) :=  \left(\frac 1T \res_{t}  \la f_0,g_0\ra+ \sum_{i = 1}^p \res_{t-x_i} \la f_{x_i},  g_{x_i} \ra 
   - \frac 1 T \res_{t^{-1}} t^2 \la f_\8, g_\8 \ra\right)\mathbf 1
\ee
where $f = (f_\8;f_{x_1},\dots,f_{x_p};f_0)$ and $g = (g_\8;g_{x_1},\dots,g_{x_p}; g_0)$ are elements of the Lie algebra \eqref{dH}. 
Equivalently, $\Heis_{\8, p, 0}$ is the quotient of the direct sum $\Heis_{\8}^{\Gamma,\op}\oplus \bigoplus_{i=1}^p  \Heis_{x_i} \oplus \Heis_{0}^\Gamma$ by the ideal generated by $\mathbf 1_{x_i} - T \mathbf 1_{0}$, $i=1,\dots,p$, and $\mathbf 1_{\8}-\mathbf 1_{0}$. Then $\mathbf 1 := \mathbf 1_{x_i}$.

Define also
\be \h_{\8,p,0} := (\h \ox \C((t^{-1})))^{\Gamma,0} \oplus \bigoplus_{i=1}^p \h \ox \C((t-x_i)) \oplus (\h \ox \C((t)))^{\Gamma,0} \nn\ee
Let us give a set of explicit generators for this commutative Lie algebra $\h_{\8,p,0}$.
Let $H_{(k,a)}$, $a=1,\dots,\dim(\Pi_k\h)$, be a basis of $\Pi_k \h$ for each $k\in \Z/T\Z$. Then $\h_{0,p,\8}$ is the commutative Lie algebra with basis  
\begin{align} 
b_j[n]_{x_i} &:= H_j \ox (t-x_i)^n \in \h\ox \C((t-x_i)),\nn\\
b_{(k,a)}[nT+k]_{0} &:= H_{(k,a)} \ox t^{nT+k} \in (\h\ox \C((t)))^{\Gamma,0},\\
b_{(k,a)}[nT+k]_{\8} &:= H_{(k,a)} \ox t^{nT+k} \in (\h\ox \C((t^{-1})))^{\Gamma,0},\nn
\end{align}
for $j \in I$, $k \in \Z/T\Z$, $a = 1, \ldots, \dim(\Pi_k\h)$ and $n \in \Z$.

\subsection{Wakimoto modules at the marked points}

For each $i=1,\dots,p$, let $\C \wac_{x_i}$ denote the one-dimensional left module over $U((\n_\C \oplus \n_\C^*)\ox \C[[t-x_i]] \oplus \C \mathbf 1_{x_i})$ on which $\mathbf 1_{x_i}$ acts as $1$ and $(\n_\C \oplus \n_\C^*)\ox \C[[t]]$ acts as zero. Define $\Mh_{x_i}$ to be the induced module over $\Heis_{x_i}$, 
\be \Mh_{x_i} := U(\Heis_{x_i}) \ox_{U((\n_\C \oplus \n_\C^*)\ox \C[[t - x_i]] \oplus \C \mathbf 1_{x_i)}} \C \wac_{x_i}.\nn\ee
Suppose we are given an $\h^*$-valued Laurent series $\nu_i \in \h^*\ox \C((t-x_i))$ in the local coordinate $t-x_i$ about the point $x_i$. Let $\C v_{\nu_i}$ denotes the one-dimensional module over $\h\ox\C((t-x_i))$ on which $f. v_{\nu_i} = v_{\nu_i} \res_{t-x_i} \nu_i(f)$, for any $f \in \h\ox\C((t-x_i))$.
Then the \emph{Wakimoto module} $W_{\nu_i}$ is the module over $\Heis_{x_i}\oplus \h\ox \C((t-x_i))$ given by
\be W_{\nu_i} := \Mh_{x_i} \ox \C v_{\nu_i}.\nn\ee
Explicitly, $W_{\nu_i}$ is the Fock module generated by a vacuum vector $\wac_{x_i}$ such that
\be a_{\alpha}[n] \wac_{x_i} = 0 , \quad n\geq 0,\qquad 
a^*_{\alpha}[n] \wac_{x_i} = 0 , \quad n\geq 1,\nn\ee
and
\be b_k[n] \wac_{x_i} = \wac_{x_i} \,\nu_{i,-n-1}(H_k)\nn\ee
where $\nu_i(t-x_i) =: \sum_{s=-S}^\8 \nu_{i,s} (t-x_i)^s$ for some $S \in \Z$ and $\nu_{i,s} \in \h^\ast$. Here and in what follows we use the obvious shorthand $a_{\alpha}[n] \wac_{x_i}$ to denote $a_{\alpha}[n]_{x_i} \wac_{x_i}$, etc.
 
Similarly, let $\C \wac_0$ denote the one-dimensional module over $(\n_\C \ox \C[[t]])^{\Gamma,0} \oplus (\n_\C^* \ox \C[[t]])^{\Gamma,-1} \oplus \C \mathbf 1_{0}$ on which $\mathbf 1_{0}$ acts as $\frac 1 T$ and the first two summands act as zero. Define  $\Mh_{0}^\Gamma$ to be the induced module over $\Heis_{0}^\Gamma$:
\be \Mh_{0}^\Gamma := U(\Heis_{0}^\Gamma) \ox_{U((\n_\C \ox \C[[t]])^{\Gamma,0} \oplus (\n_\C^* \ox \C[[t]])^{\Gamma,-1} \oplus \C \mathbf 1_{0})} \C \wac_0. \nn\ee
Suppose we are given an element $\nu_0\in (\h^*\ox \C((t)))^{\Gamma,-1}$. That is, $\nu_0$ is a $\h^*$-valued Laurent series in $t$ such that $\nu_0(\omega t) = \omega^{-1} \sigma \nu_0(t)$.
Let $\C v_{\nu_0}$  denote the one-dimensional module over $(\h\ox \C((t)))^{\Gamma,0}$ given by
\be f. v_{\nu_0} = v_{\nu_0} \res_{t} \nu_0(f),\ee
for any $f \in (\h\ox \C((t)))^{\Gamma,0}$. We may then consider the \emph{twisted Wakimoto module} $W^\Gamma_{\nu_0}$ defined as the $\Heis_{0}^\Gamma\oplus (\h\ox \C((t)))^{\Gamma,0}$-module
\be W_{\nu_0}^\Gamma := \Mh_{0}^\Gamma \ox \C v_{\nu_0}.\nn\ee
Explicitly, it is the Fock module generated by a vacuum vector $\wac_{0}$ such that $1\wac_{0} = \wac_{0}\frac 1T$,
\be a_{(k,\alpha)}[n] \wac_{0} = 0,  \quad n\geq 0,\qquad
  a^*_{(k,\alpha)}[n] \wac_{0} = 0 , \quad n\geq 1,\nn\ee
and
\be b_{(k,a)}[n] \wac_{0} = \wac_{0} \,\nu_{0,-n-1}(H_{(k,a)})\nn\ee
where $\nu_0(t) =: \sum_{s=-S}^\8 \nu_{0,s} t^s$ for some $S \in \Z$ and $\nu_{0, s} \in \h^\ast$.

Finally, define $$\Heis_{\8}^{\Gamma,+} :=  (\n_\C\ox t^{-1}\C[[t^{-1}]])^{\Gamma,0}\oplus(\n^*_\C\ox t^{-1}\C[[t^{-1}]])^{\Gamma,-1}\oplus \C\mathbf 1_{\8},$$ 
and let $\Mh^{\Gamma,\vee}_{\8}$ denote the right module over $U(\Heis_{\8}^\Gamma)$ induced from the trivial one-dimensional right module $\C\lwac$ over $U(\Heis_{\8}^{\Gamma,+})$ on which $\mathbf 1_{\8}$ acts as $\frac 1T$ and the first two summands act as zero:
\be \Mh_{\8}^{\Gamma,\vee} := \C \lwac\ox_{U(\Heis_{\8}^{\Gamma,+})}  U(\Heis_{\8}^\Gamma). \nn\ee
Suppose $\nu_\8\in (\h^*\ox \C((t^{-1})))^{\Gamma,-1}$. That is, $\nu_\8$ is a $\h^*$-valued Laurent series in $t^{-1}$ such that $\nu_\8(\omega^{-1} t^{-1}) = \omega^{-1} \sigma \nu_\8(t^{-1})$.
Let $\C v^*_{\nu_\8}$ denote the one-dimensional module over $(\h\ox \C((t^{-1})))^{\Gamma,0}$ given by
\be f. v^*_{\nu_\8} = v^*_{\nu_\8} \res_{t^{-1}} (-t^2)\nu_\8(f),\ee
for $f \in (\h\ox \C((t^{-1})))^{\Gamma,0}$.
Then we have the right $U(\Heis_{\8}^\Gamma\oplus (\h\ox \C((t^{-1})))^{\Gamma,0})$-module
\be W_{\nu_\8}^{\Gamma,\vee} := \Mh_{\8}^{\Gamma,\vee} \ox \C v^*_{\nu_\8}.\nn\ee
Explicitly, $W_{\nu_\8}^{\Gamma,\vee}$ is the Fock module generated by a vacuum vector $\lwac$ such that
\be \lwac a_{(k,\alpha)}[n]_{\8} = 0, \quad n< 0,\qquad
    \lwac a^*_{(k,\alpha)}[n]_{\8}  = 0 , \quad n\leq 0,\nn\ee
and
\be \lwac b_{(k,a)}[n]_{\8} = - \nu_{\8,-n-1}(H_{(k,a)}) \lwac \,\label{lwacdef}\ee
where $\nu_\8(t) =: \sum_{s=-\8}^S \nu_{\8,s} t^s$ for some $S \in \Z$ and $\nu_{\8, s} \in \h^\ast$.

\subsection{Free field realization}
The modules $W_{\nu_i}$ are \emph{smooth}. That means, by definition, that for each $v\in W_{\nu_i}$
\be 0=a_\alpha[n]_{x_i} v =a^*_{\alpha}[n]_{x_i} v = b_k[n]_{x_i} v \quad\text{for all}\quad n\gg 0,\nn\ee
for all $\alpha\in \Delta^+$ and $k\in I$.\footnote{That is, for each $v\in W_{\nu_i}$ $\alpha\in \Delta^+$ and $k\in I$ there exists an $n\in \Z$ such that $0=a_\alpha[m]_{x_i} v =a^*_{\alpha}[m]_{x_i} v = b_k[m]_{x_i} v$ for all $m\geq n$.}
Similarly $W_{\nu_0}^\Gamma$ is smooth.
The module $W_{\nu_\8}^{\Gamma,\vee}$ is \emph{co-smooth}. By that we mean that for each $v\in W_{\nu_\8}^{\Gamma,\vee}$,
\be 0= v\, a_{(k,\alpha)}[n]_{\8} = v\, a^*_{(k,\alpha)}[n]_{\8} = v\, b_{(k,a)}[n]_{\8} \quad\text{for all}\quad n\ll 0, \nn\ee
for all $k\in \ZT$, $\alpha \in \Delta^+_k$ and $a\in\{1,\dots,\dim\Pi_k\h\}$.

Let now $\pi_0\simeq  \C[ b_i[n]]_{i\in I;\,n\leq -1}$ be the induced representation of $\h\ox\C((t))$ in which $b_i[n]$ acts as $0$ for all $i\in I$ and all $n\in \Z_{\geq 0}$. Let $\Mh$ be the induced module over the Heisenberg Lie algebra $\Heis(\g)$ with generators $a_\alpha[n]$, $a^*_\alpha[n]$.
Then
\be\WW_0:= \Mh\ox \pi_0,\label{WW0def}\ee 
is an induced representation of $\Heis(\g)\oplus\h\ox\C((t))$. 
Explicitly, $\WW_0$ is the Fock module generated by a vacuum vector $\wac$ such that
\be a_{\alpha}[n] \wac = 0 , \quad n\geq 0,\qquad 
a^*_{\alpha}[n] \wac = 0 , \quad n\geq 1,\nn\ee
and
\be b_k[n]\wac = 0, \quad n\geq 0.\nn\ee
There is a $\Z$-grading on $\WW_0$ defined by $\deg\wac=0$ and $\deg a_\alpha[n] = \deg a^*_\alpha[n] = \deg b_i[n]= n$.

Let us recall some facts about the free-field realization of $\gh$. 

The $\Heis(\g)\oplus\h\ox\C((t))$-module $\WW_0$ is endowed with the structure of a \emph{vertex algebra}. In the notation of \cite{VY2}, we have the ``big'' Lie algebra $\U(\WW_0) := \Lie_{\C((t))} \WW_0$ of all formal modes of states in $\WW_0$. For each of the marked points $x_i$, $i=1,\dots, p$, there is a ``local'' copy $\U(\WW_0)_{x_i}:= \Lie_{\C((t-x_i))}\WW_0$ of $\U(\WW_0)$. It contains $\Heis_{x_i} \oplus \h\ox\C((t-x_i))$ as a subalgebra. 
Moreover, every smooth module over $\Heis_{x_i}\oplus\h\ox\C((t-x_i))$ on which $\bm 1_{x_i}$ acts as $1$ becomes a smooth $\U(\WW_0)_{x_i}$-module in a canonical way. See \emph{e.g.} \cite[Proposition 5.9]{VY2}.

At the fixed-points $0$ and $\8$ we have the local copies $\U(\WW_0)_0 := \Lie_{\C((t))} \WW_0$ and $\U(\WW_0)_\8 := \Lie_{\C((t^{-1}))} \WW_0$ of the big Lie algebra $\U(\WW_0)$. 
The automorphism $\sigma$ of $\Heis(\g)\oplus\h\ox\C((t))$ extends in a unique way to an automorphism of $\WW_0$ as a vertex algebra. So we get an action by automorphisms of the group $\Gamma$ on both of these local Lie algebras. Let $\U(\WW_0)_0^\Gamma$ and $\U(\WW_0)^\Gamma_\8$ denote the respective fixed-point subalgebras. We have embeddings of Lie algebras $\Heis_{0}^\Gamma \oplus(\h\ox\C((t)))^{\Gamma,0} \into \U(\WW_0)_0^\Gamma$ and $\Heis_{\8}^\Gamma \oplus(\h\ox\C((t^{-1})))^{\Gamma,0} \into \U(\WW_0)^\Gamma_\8$.  

Every smooth module over $\Heis_{0}^\Gamma \oplus(\h\ox\C((t)))^{\Gamma,0}$ on which $\bm 1_0$ acts as $\frac 1 T$ becomes a smooth module over $\U(\WW_0)_0^\Gamma$. See \cite[Proposition 5.8]{VY2}.
Every co-smooth module over $\Heis_{\8}^\Gamma \oplus(\h\ox\C((t^{-1})))^{\Gamma,0}$ on which $\bm 1_\8$ acts as $\frac 1 T$ becomes a co-smooth module over $\U(\WW_0)_\8^{\Gamma}$. See Proposition \ref{prop: big co-smooth}.

Now, the vacuum Verma module $\Vcrit$ also has the structure of a vertex algebra. Its ``big'' Lie algebra $\U(\Vcrit)$ contains $\gh$ as a subalgebra, via the embedding sending $K\mapsto (-h^\vee \vac)_{[-1]}$ and $A[n]\mapsto (A[-1]\vac)_{[-1]}$ for $A\in \g$. In this way smooth modules over $\U(\Vcrit)$ pull back to smooth modules over $\gh$ of level $-h^\vee$.

The subalgebra $\U(\Vcrit)^\Gamma$ contains the twisted affine algebra $\gh^\Gamma$ as a subalgebra.

The \emph{Feigin-Frenkel homomorphism} is a homomorphism of vertex algebras $\Vcrit \to \WW_0$. It induces homomorphisms of the big Lie algebras, $\U(\Vcrit)_{x_i} \to \U(\WW_0)_{x_i}$, for every $i=1,\dots,p$, and $\U(\Vcrit)_{0}^\Gamma \to \U(\WW_0)_{0}^\Gamma$ and $\U(\Vcrit)_\8^\Gamma \to \U(\WW_0)_\8^\Gamma$. 

By means of this homomorphism, every smooth module over $\Heis_{x_i}\oplus\h\ox\C((t-x_i))$ on which $\bm 1_{x_i}$ acts as $1$ becomes a smooth module of level $-h^\vee$ over the local copy $\gh_{x_i}$ of $\gh$ at $x_i$, for each $i=1,\dots, p$. Every smooth module over $\Heis_{0}^\Gamma \oplus(\h\ox\C((t)))^{\Gamma,0}$ becomes a smooth module of level $-h^\vee/T$ over $\gh_0^\Gamma$. And every co-smooth module over $\Heis_{\8}^{\op,\Gamma} \oplus(\h\ox\C((t^{-1})))^{\Gamma,0}$ becomes a co-smooth module over $\gh_\8^{\op,\Gamma}$ of level $-h^\vee/T$.

In particular, these statements apply to the Wakimoto modules $W_{\nu_i}$, $i=1,\dots,p$, $W_{\nu_0}^\Gamma$ and $W_{\nu_\8}^{\Gamma,\vee}$. We shall need the following facts about the structure of these modules. 

For each $i=1,\dots,p$, let $\widetilde W_{\nu_i}$ denote the subspace of grade zero, \ie the linear  span of states of the form $a^*_{\alpha_1}[0] \dots a^*_{\alpha_k}[0] \wac_{x_i}$, $k \in \Z_{\geq 0}$. This subspace $\widetilde W_{\nu_i}$ becomes a left module over the subalgebra $U(\g)\subset U(\gh_{x_i})$ generated by zero-modes, $X[0]_{x_i}$ with $X\in \g$.
Suppose $\nu_i$ has at most a simple pole. Then there is \cite{FFR} an isomorphism of left $U(\g)$-modules 
\be\widetilde W_{\nu_i} \cong M^*_{\res_t\nu_i}.\label{wmi}\ee

Let $\widetilde W_{\nu_0}^\Gamma\subset W_{\nu_0}^\Gamma$ denote the subspace of grade zero, \ie the linear span of states of the form $a^*_{(0,\alpha_1)}[0] \dots a^*_{(0,\alpha_k)}[0] \wac_0$, $k\in \Z_{\geq 0}$, $\alpha_1, \dots, \alpha_k\in \Delta_0^+$. It is a left module over $U(\g^\sigma)$. Suppose $\nu_0$ has at most a simple pole. Then there is \cite[Proposition 4.4]{VY1} an isomorphism of left $U(\g^\sigma)$-modules
\be \widetilde W_{\nu_0}^\Gamma \cong M^{*,\sigma}_{\frac1T (\res_t(\nu_0) - \Lambda_0)}\label{wmig}\ee
where for $\lambda\in \h^{*,\sigma}$ we denote by $M^{*,\sigma}_{\lambda}$ the contragredient Verma module over $\g^\sigma$ of highest weight $\lambda$, and where $\Lambda_0$ is the weight given in \eqref{l0def}. 

Let $\C_\chi$ denote the one-dimensional representation of $(\g^\op\ox t^{-1} \C[[t^{-1}]])^\Gamma$ defined by
\begin{equation*}
(\g^\op\ox t^{-1} \C[[t^{-1}]])^\Gamma \longrightarrow (\g^\op\ox t^{-1} \C[[t^{-1}]])^\Gamma/ (\g^\op\ox t^{-2} \C[[t^{-1}]])^\Gamma \simeq_\C \Pi_{-1} \g \overset{\chi}\longrightarrow \C,
\end{equation*}
where $\chi$ is defined as in \S\ref{sec: BAE}.

The module $W^{\Gamma,\vee}_{\nu_\8} = \Mh^{\Gamma,\vee}\ox \C v^*_{\nu_\8}$ is a right module over $U(\gh^\Gamma)$ of level $-h^\vee/T$. 
Let us suppose that 
\be \nu_\8(t) = - \chi + \mc O(1/t) \nn\ee
(\ie we suppose that $t^2\nu_\8$ has at most a double pole in $t^{-1}$).  Then it follows from the explicit form of the Feigin-Frenkel homomorphism and the definition of the quasi-module map $Y_W$ that the vector $\lwac$ of \eqref{lwacdef} obeys
\be \lwac  (\g\ox t^{-2}\C[[t^{-1}]])^{\Gamma} =0, \qquad \lwac (\Pi_{1}A)[-1] = \frac 1T\chi(A) \lwac. \nn\ee
In other words there is an isomorphism of $(\g^\op\ox t^{-1} \C[[t^{-1}]])^\Gamma$-modules 
\be \C \lwac \cong \C_\chi. \label{wmigg}\ee 

\subsection{The generators $G_\alpha[n]$.}\label{Gsec}
The Verma module $M_\lambda = U(\g)\ox_{U(\h \oplus \n)} \C v_\lambda \cong_\C U(\n_-) \ox_\C \C v_\lambda$,  \S\ref{sec: cd}, is by definition a left module over $U(\g)$. In particular it is a left module over the subalgebra $U(\n_-)\subset U(\g)$. But $M_\lambda$ also admits a second left action, call it $\triangleright$, of this subalgebra $U(\n_-)$. These two left actions of $U(\n_-)$ are mutually commuting. They are given by 
\be X \on (n \ox v_\lambda) = Xn \ox v_\lambda\quad\text{ and }\quad X\triangleright (n\ox v_\lambda) = - nX \ox v_{\lambda}, \qquad X\in \n_-, n\in U(\n_-).\nn\ee
In particular 
\be X \on (1\ox v_\lambda) = - X\triangleright (1\ox v_\lambda).\label{GEm}\ee

Correspondingly the contragredient Verma module admits a second left action of $U(\n)$. We write $G_{\alpha}$, $\alpha\in \Delta^+$, for the generators of this second copy of $\n$. So we have $[G_\alpha,E_\beta] = 0$ for all $\alpha,\beta\in \Delta^+$, where $E_\alpha$ are the generators of $\n$ from \S\ref{sec: cd}.

By means of the free-field realization, the Wakimoto module $W_\nu$ becomes a module not only over $\gh$ but also over a copy of $\n\ox \C((t))$ whose generators we denote $G_{\alpha}[n] := G_\alpha\ox t^n$.  In particular $\widetilde W_\nu$ is a module over the subalgebra $U(\n)\subset U(\n\ox \C((t)))$ generated by zero modes. Then \eqref{wmi} is an isomorphism also of modules over this copy of $\n$, \ie $G_\alpha[0]$ acts as $G_\alpha$.

The analogous statements hold for $\n^\sigma$ too. Namely we have generators $G_{(0,\alpha)}$, $\alpha\in \Delta^+_0$, of a second copy of $U(\n^\sigma)$ which acts from the left on contragredient Verma modules $M_\lambda^{*,\sigma}$ over $\n^\sigma$. These generators commute with the generators $E_{(0,\alpha)}$ of the standard left action of $U(\n^\sigma)$. 
And, arguing for $G_{(0,\alpha)}$ as for $E_{(0,\alpha)}$ in the proof of \cite[Proposition 4.4]{VY1}, one checks that \eqref{wmig} an isomorphism of modules over this second copy of $\n^\sigma$.\footnote{Namely, one checks that  
\be (G_{(0,\alpha)}[-1]\wac)^{W}_{(0)} = \sum_{\beta\in \Delta^+_0} R_{(0,\alpha)}^{(0,\beta)}(Ta^*_{(0,\bullet)}[0]) a_{(0,\beta)}[0]\ee
which confirms that $(G_{(0,\alpha)}[-1]\wac)^{W}_{(0)}$ is nothing but the (usual, untwisted) free-field realization of $\n^\sigma_{\langle G\rangle}$.}

\subsection{Global Heisenberg algebra and coinvariants}
Let $\Heis_\xsa^\Gamma$ be the commutative Lie algebra
\be \Heis_\xsa^\Gamma := (\n_\C \ox \C_{\8,\Gamma\bm x,0}(t))^{\Gamma,0} \oplus (\n_\C^*\ox \C_{\8,\Gamma\bm x,0}(t))^{\Gamma,-1}. \nn\ee 
That is, an element of $\Heis_\xsa^\Gamma$ is a pair $(f(t),g(t))$ where $f(t)$ is a rational function valued in $\n_\C$ with poles at most at the points $0,x_1,\dots,x_p,\8$ and obeying the equivariance condition $f(\omega t) = \sigma f(t)$, and where $g(t)$ is a rational function valued in $\n^*_\C$ with poles at most at the points $0,x_1,\dots,x_p,\8$ and obeying the equivariance condition $g(\omega t) = \omega^{-1} \sigma g(t)$.
By virtue of the residue theorem there is an embedding of Lie algebras,
\be \Heis_\xsa^\Gamma \longinto \Heis_{\8,p,0} ,\nn\ee
by taking Laurent expansions, $f(t) \mapsto (-\iota_{t^{-1}}; \iota_{t-x_1},\dots,\iota_{t-x_p}; \iota_{t})f(t)$. Cf. \eqref{iotamap}.

Let us write
\be \h_\xsa^{\Gamma} := (\h \ox \C_{0,\Gamma\bm x,\8}(t))^{\Gamma,0}\quad\text{and}\quad\h_\xsa^{*,\Gamma} :=  
\left(\h^*\ox \C_{0,\Gamma\bm x,\8}(t)\right)^{\Gamma,-1}.\nn\ee 
There is an embedding of commutative Lie algebras
\be \h_\xsa^{\Gamma} \longinto \h_{\8,p,0} \nn\ee
and we have the following, which is \cite[Proposition 4.3]{VY1} but now including the pole at $\8$.

\begin{lem}\label{prop:vh} The space of coinvariants
$\left.\C v_{\nu_\8} \ox \bigotimes_{i=1}^p \C v_{\nu_i}\ox \C v_{\nu_0}\right/ \h_\xsa^\Gamma$
is one-dimensional if and only if there exists a $\nu(t)\in \h_\xsa^{*,\Gamma}$ such that 
\be (\nu_\8; \nu_1,\dots,\nu_p; \nu_0) = (-\iota_{t^{-1}}; \iota_{t-x_1},\dots,\iota_{t-x_p}; \iota_t) \nu(t). \nn\ee
Otherwise it is zero-dimensional. \qed
\end{lem}

Define
\begin{multline} \Heis_{\8,p,0}^+ := (\n_\C \ox \C[[t^{-1}]])^{\Gamma,0} \oplus (\n^*_\C\ox \C[[t^{-1}]])^{\Gamma,-1} \\
{}\oplus 
\bigoplus_{i=1}^p (\n_\C\oplus \n^*_\C)\ox \C[[t-x_i]]\\
\oplus
 (\n_\C \ox \C[[t]])^{\Gamma,0} \oplus (\n^*_\C\ox \C[[t]])^{\Gamma,-1}\oplus \C\mathbf 1
\nn\end{multline}
Then 
\be \Heis_{\8,p,0} = \Heis_{\8,p,0}^+ + \Heis_\xsa^\Gamma \qquad 
  \text{and} \qquad \Heis_{\8,p,0}^+ \cap \Heis_\xsa^\Gamma  = \{0\}.\nn\ee
We may regard $M^{\Gamma,\vee}_{\8}$ as a left module over $U(\Heis_{\8}^{\Gamma,\op})$ and then
\be \left(\Mh^{\Gamma,\vee}_{\8} \ox \bigotimes_{i=1}^p \Mh_{x_i} \ox \Mh^{\Gamma}_{0}\right) = U(\Heis_{\8,p,0}) \ox_{U(\Heis_{\8,p,0}^+)} \C \lwac \ox \wac_{x_1} \ox \dots \ox \wac_{x_p} \ox \wac_0 \ee

Hence 
\be \left(\Mh^{\Gamma,\vee}_{\8} \ox \bigotimes_{i=1}^p \Mh_{x_i} \ox \Mh^{\Gamma}_{0}\right)
\big/ \Heis_\xsa^\Gamma \cong_\C \C \lwac \ox \wac_{x_1} \ox \dots \ox \wac_{x_p} \ox \wac_0 \cong_\C \C\nn.\ee

Therefore for any $\nu(t) \in \h_\xsa^{*,\Gamma}$ the space of coinvariants
\be \left(W^{\Gamma,\vee}_{-\iota_{t^{-1}}\nu(t)} \ox \bigotimes_{i=1}^p W_{\iota_{t-x_i} \nu(t)} \ox W^\Gamma_{\iota_t \nu(t)} \right)\big/ \left(\Heis_\xsa^\Gamma \oplus \h_\xsa^\Gamma\right)\ee
has dimension one. That means there is a unique $\left(\Heis_\xsa^\Gamma \oplus \h_\xsa^\Gamma\right)$-invariant linear functional, call it  $\tau_{\nu(t)}$, 
\be \tau_{\nu(t)} : 
 W^{\Gamma,\vee}_{-\iota_{t^{-1}}\nu(t)} \ox \bigotimes_{i=1}^p W_{\iota_{t-x_i} \nu(t)} \ox W^\Gamma_{\iota_t \nu(t)} \longrightarrow \C\nn\ee
normalised such that $\tau_{\nu(t)}(\lwac \ox \wac_{x_1} \ox \dots \ox \wac_{x_p} \ox \wac_0 )=1$.  

By functoriality, \emph{i.e.} \cite[Theorem 6.2]{VY2} and \cite[Corollary 6.6]{VY2} generalised to the situation of Appendix \ref{app: YWmap} for coinvariants of a tensor product of modules including one attached to infinity, the functional $\tau_{\nu(t)}$ is also invariant under $\g_\xsa^\Gamma$.

\subsection{Proof of Theorem \ref{mthm}} 
Let $c(1),\dots,c(m)\in I$, $\lambda_1,\dots,\lambda_N\in \h^*$, $\lambda_0\in \h^{*,\sigma}$ and $\chi$ be as in \S\ref{sec: mr}. Recall 
\be (x_1,\dots,x_p) = (z_1,\dots,z_N,w_1,\dots,w_m).\nn\ee 
We now fix
\be \nu(t) := \chi + \sum_{r\in \Z_T} \left(\sum_{i=1}^N\frac{\lsigma^r\lambda_i}{t-\omega^rz_i} - \sum_{j=1}^m \frac{\lsigma^r\alpha_{c(j)}}{t-\omega^rw_j}\right) + \frac{T\lambda_0+\Lambda_0}{t} \in
\h_\xsa^{*,\Gamma} ,\label{chitdef}\ee
where $\Lambda_0$ is as in \eqref{lambda0def}.

Let $\nu_j^{0}$ be the constant term in the Laurent expansion of $\nu(t)$ about $w_j$, for $j=1,\dots,m$. Note that the Bethe equations \eqref{tbe} are equivalent to the statement that
\be \la \nu_j^{0},\alpha_{c(j)} \ra = 0 ,\quad j=1,\dots,m. \ee 
Under this condition the vector $G_{c(j)}[-1]\wac_{w_j} \in W_{\res_{t-w_j}\nu(t)}$ is singular for $\gh$ (\ie $X[n] G_{c(j)}[-1]\wac_{w_j} = 0$ for all $n\geq 0$ and all $X\in \g$) \cite{FFR}. 
Henceforth, suppose that the Bethe equations are indeed satisfied. Then we get a linear functional invariant under the Lie subalgebra $\g_\zsa^\Gamma\subset \g_\xsa^\Gamma$,
\be \psi_{\nu(t)}  : W^{\Gamma,\vee}_{-\iota_{t^{-1}}\nu(t)} \ox \bigotimes_{i=1}^N W_{\iota_{t-z_i} \nu(t)} \ox W^\Gamma_{\iota_t \nu(t)}\to \C\label{psidef}\ee
defined by
\be \psi_{\nu(t)}(v_\8,v_1,\dots,v_N,v_0) = \tau_{\nu(t)}\big( v_\8,v_1,\dots,v_N,G_{c(1)}[-1]\wac_{w_1},\dots,G_{c(m)}[-1]\wac_{w_m}, v_0 \big).\nn\ee


Now we use the set-up of \S\ref{sec: indg} and specialize by setting
\begin{alignat}{3} \M_{z_i} &= M_{\lambda_i}^* && \cong \widetilde W_{\res_{t-z_i}\nu(t)} &&\subset W_{\iota_{t-z_i}\nu(t)} , \qquad i=1,\dots,N, \nn\\
 \M_0 &= M^{*,\sigma}_{\lambda_0} && \cong \widetilde W^\Gamma_{\res_t \nu(t)} &&\subset W^\Gamma_{\iota_t \nu(t)}, \nn\\
  \M_\8 &= \C_\chi &&\cong \C\lwac &&\subset W^{\Gamma,\vee}_{-\iota_{t^{-1}} \nu(t)}.\label{mmods}\end{alignat}
Here the isomorphisms are as in \eqref{wmi}, \eqref{wmig}, \eqref{wmigg}. 

Recall the homogeneous gradation of $\gh$ (namely $\deg X[n] = n$).
\begin{lem} Let $W$ be a graded $\gh$-module, $W= \bigoplus_{n\in \Z} W_n$. Then the subspace $W_{\geq 0}:= \bigoplus_{n=0}^\8 W_n$ is a module over $\g\ox \C[[t]]\oplus \C K \subset \gh$ and there is a canonical map of $\gh$-modules
\begin{equation*}
U(\gh) \ox_{U(\g\ox \C[[t]] \oplus \C K)} W_{\geq 0} \longrightarrow W; \quad x\ox v \longmapsto x\on v.
\end{equation*}
\qed
\end{lem}
By virtue of this lemma there are canonical homomorphisms 
\begin{align} \MM^{-h^\vee}_{z_i} &\longrightarrow W_{\iota_{t-z_i}\nu(t)}, \quad i=1,\dots,N,\nn\\
         \MM^{-h^\vee/T}_0 &\longrightarrow  W^\Gamma_{\iota_t \nu(t)},\nn\end{align}
of, respectively  $\g_{z_i}$-modules, $i=1,\dots,N$ and of $\gh^\Gamma_{0}$-modules.
Similarly there is a canonical homomorphism
\be          \MM^{-h^\vee/T}_\8 \longrightarrow W^{\Gamma,\vee}_{-\iota_{t^{-1}} \nu(t)} \nn\ee
of $\gh^\Gamma_{\8}$-modules.
So in total we get a canonical map of $\gh_{0,p,\8}$-modules
\be \MM^{-h^\vee/T}_\8 \ox \bigotimes_{i=1}^N \MM^{-h^\vee}_{z_i} \ox \MM^{-h^\vee/T}_0 \longrightarrow W^{\Gamma,\vee}_{-\iota_{t^{-1}}\nu(t)} \ox \bigotimes_{i=1}^N W_{\iota_{t-z_i} \nu(t)} \ox W^\Gamma_{\iota_t \nu(t)}.\ee
In particular it is a map of $\g^\Gamma_\zsa\subset \g^\Gamma_\xsa \subset \gh_{\8,p,0}$-modules. 
Composing this map with the $\g^\Gamma_\zsa$-invariant functional $\psi_{\nu(t)}$ of \eqref{psidef} we obtain a $\g^\Gamma_\zsa$-invariant functional $\MM^{-h^\vee/T}_\8 \ox \bigotimes_{i=1}^N \MM^{-h^\vee}_i \ox \MM^{-h^\vee/T}_0 \to \C$. 
By Proposition \ref{prop:isom1} this is the same thing as a functional 
\be \C_\chi  \ox \bigotimes_{i=1}^N M_{\lambda_i}^*\ox M_{\lambda_0}^{*,\sigma}  \cong \bigotimes_{i=1}^N M_{\lambda_i}^*\ox M_{\lambda_0}^{*,\sigma} \longrightarrow \C.\label{lf}\ee 
(Indeed, the latter map is nothing but the restriction of $\psi_{\nu(t)}$ to the subspace $\C_\chi \ox \bigotimes_{i=1}^N M_{\lambda_i}^*\ox M_{\lambda_0}^{*,\sigma}\subset \MM^{-h^\vee/T}_\8 \ox \bigotimes_{i=1}^N \MM^{-h^\vee}_{z_i} \ox \MM^{-h^\vee/T}_0$.) 
Such functional is the same thing as a vector 
\be \psi_{\nu(t)} \in \bigotimes_{i=1}^N M_{\lambda_i}\ox M_{\lambda_0}^\sigma. \label{bv}\ee
The proof that this vector is an eigenstate of the Gaudin algebra with the given eigenvalue for the quadratic Hamiltonians proceeds as in \cite{VY1,FFT} following \cite{FFR}. 

Now let us show that this vector $\psi_{\nu(t)}$ of \eqref{bv} is given by the recursive definition \eqref{psidef}. (In particular, that it depends on $\chi$ only through the Bethe equations.)


Let $v_\lambda\in M_\lambda$ be a highest weight vector. By the pairing $M_\lambda \ox M^*_\lambda \to \C$ we can regard $v_\lambda$ as a linear map $M^*_\lambda \to \C$. 

For notational convenience let us reorder tensor factors in the argument of $\tau_{\nu(t)}$ in such a way that
\be \tau_{\nu(t)} : W^{\Gamma,\vee}_{-\iota_{t^{-1}}\nu(t)} \ox \bigotimes_{i=1}^N W_{\iota_{t-z_i} \nu(t)} \ox W^\Gamma_{\iota_t \nu(t)} \ox \bigotimes_{j=1}^m W_{\iota_{t-w_j} \nu(t)} \longrightarrow \C \ee

Let $v_\chi \ox \bm v \ox v_0\in  \widetilde W^{\Gamma,\vee}_{-\iota_{t^{-1}}\nu(t)} \ox \bigotimes_{i=1}^N \widetilde W_{\iota_{t-z_i} \nu(t)} \ox \widetilde W^\Gamma_{\iota_t \nu(t)}$.
\begin{lem} For any  $X_1,\dots,X_{s}\in \n_{\langle G\rangle}$, we have
\begin{align} 
&\tau_{\nu(t)}(v_\8,\bm v,v_0,X_{1}[-1] \wac_{w_1}, \dots,X_{s-1}[-1]\wac_{w_{s-1}}, X_s[-1]\wac_{w_s},\wac_{w_{s+1}},\dots,\wac_{w_m}) \nn\\
&\quad =  \sum_{i=1}^N\sum_{k\in \Z_T}\frac{\tau (v_\8,(\sigma^kX_{s})[0]_{z_i} \bm v , v_0, X_{1}[-1]\wac_{w_1},\dots,X_{s-1}[-1]\wac_{w_{s-1}},\wac_{w_s},\dots,\wac_{w_m})}{w_s-\omega^{-k}z_i} \nn\\
&\qquad + \frac{\tau (v_\8,\bm v, (T\Pi_0 X_{s})_{0}^W  v_0, X_{1}[-1]\wac_{w_1},\dots,X_{s-1}[-1]\wac_{w_{s-1}},\wac_{w_s},\dots,\wac_{w_m})}{w_s} \nn\\
&\qquad + \sum_{j=1}^{s-1} \sum_{k\in \Z_T} \frac{1}{w_s-\omega^{-k}w_j} 
  \tau_\Gamma(v,X_{1}[-1]\wac_{w_1},\dots,X_{j-1}[-1]\wac_{w_{j-1}}, 
[\sigma^k(X_{s}),X_{j}][-1]\wac_{w_j},\nn\\
&\qquad\qquad\qquad\qquad\qquad\qquad\qquad\qquad X_{j+1}[-1]\wac_{w_{j+1}},\dots,X_{s-1}[-1]\wac_{w_{s-1}},\wac_{w_s},\dots,\wac_{w_m}).\nn\end{align} 
\end{lem}

\begin{proof}
Let us write $\bm y = X_{1}[-1]\wac_{w_1} \ox \dots\ox X_{s-1}[-1]\wac_{w_{s-1}}\ox \wac_{w_s}\ox \dots\ox \wac_{w_m} \in\bigotimes_{j=1}^m W_{\iota_{t-w_j} \nu(t)}$ for brevity. We have
\be 
0= \left[\frac{\sigma^k X_s[-1]\wac}{\omega^{-k} t- w_s} \on (v_\8 \ox \bm v \ox v_0 \ox \bm y) \right].
\ee
Now 
\begin{multline}
\sum_{k\in \Z_T}\frac{\sigma^k X_s[-1]\wac}{\omega^{-k} t- w_s} \on (v_\8 \ox \bm v \ox v_0\ox \bm y)\\
=
- \res_{t^{-1}} t^2 \iota_{t^{-1}} \frac{1}{\omega^{-k}t-w_s} v_\8 Y_W(\sigma^k X_s[-1]\wac, t)\ox  \bm v \ox v_0\ox \bm y\\
+
v_\8 \ox \sum_{i=1}^N \res_{t-z_i} \iota_{t-z_i} \frac{1}{\omega^{-k}t-w_s} Y_M(\sigma^k X_s[-1]\wac, t-z_i)_{z_i} \bm v \ox v_0\ox \bm y
\\
+ 
v_\8 \ox \bm v \ox \res_{t} \iota_{t} \frac{1}{\omega^{-k}t-w_s} Y_W(\sigma^k X_s[-1]\wac, t) v_0\ox \bm y
\\
+
v_\8 \ox  \bm v \ox v_0 \ox \sum_{j=1}^{m} \res_{t-w_j} \iota_{t-w_j} \frac{1}{\omega^{-k}t-w_s} Y_M(\sigma^k X_s[-1]\wac, t-w_j)_{w_j} \bm y.
 \end{multline}
The $\res_{t-w_s}$ term here is  $v_\8 \ox  \bm v \ox v_0 \ox X_{1}[-1]\wac_{w_1}\ox \dots\ox X_{s-1}[-1]\wac_{w_{s-1}}\ox X_s[-1]\wac_{w_s}\ox \wac_{w_{s+1}}\ox \dots\ox \wac_{w_m}$. (The expansion of $1/(\omega^{-k} t - w_s)$ at $t=w_s$ has a regular part but it does not contribute because $(X_s[-1]\wac)_{[n]}\wac = X_s[n]\wac$ is zero for all $n\geq 0$.)   

Consider the first term on the right. Let $\widetilde W^{\Gamma,\vee}_\nu \subset W^{\Gamma,\vee}_\nu$ denote the subspace spanned by vectors of the form $\lwac a_{(0,\alpha_1)}[0] \dots a_{(0,\alpha_k)}[0]$, $k\in \Z_{\geq 0}$, $\alpha_1,\dots,\alpha_k\in \Delta_0^+$. We have
\be \res_{t^{-1}} t^2 \iota_{t^{-1}} \frac{1}{\omega^{-k}t-w_s} v_\8 Y_W(\sigma^k X_s[-1]\wac, t)
=
  \sum_{n\geq 0} w_s^n \omega^{k(n+1)} v_\8 (\sigma^k X_s[-1] \wac)^W_{(-n-1)} \label{8swap}\ee
and this vanishes for all $v_\8 \in \widetilde W^{\Gamma,\vee}_{-\iota_{t-1}\nu(t)}$, on $\Z$-grading grounds. (The free-field generators of $\n_{\langle G\rangle}$ involve $a,a^*$ but not $b$.) Consider the third term. We have
\be \res_{t} \iota_{t} \frac{1}{\omega^{-k}t-w_s} Y_W(\sigma^k X_s[-1]\wac, t) v_0
= - \sum_{n\geq 0} w_s^{-n-1} \omega^{-kn} (\sigma^k X_s[-1]\wac)^W_{(n)} v_0
\ee
and on grading grounds only the term $n=0$ in this sum contributes, since $v_0\in \widetilde W_{\iota_t \nu(t)}$. 
The remaining terms are similar.
\end{proof}

We have $\tau_{\nu(t)}(v_\8,\bm v, a^*_{(0,\alpha)}[0] v_0, \wac_{w_1},\dots,\wac_{w_m}) = 0$  for any weight $\alpha\in \Delta^+_0$. Indeed, this follows from the invariance of $\tau$ under the $\Gamma$-equivariant function
$E^*_{(0,\alpha)}\ox t^{-1} \in (\n^*_\C\ox \C_{0,\Gamma\bm x,\8}(t))^{\Gamma,-1}$. Similarly, we have
$\tau_{\nu(t)}(v_\8,a^*_\alpha[0]_{(z_i)} \bm v,v_0, \wac_{w_1},\dots,\wac_{w_m}) =0$,  $i=1,\dots,N$, for any root $\alpha\in \Delta^+$, as in \cite{VY1,FFR,ATY}.

It follows that, for any $v_\8\in\widetilde W^{\Gamma,\vee}_{-\iota_{t^{-1}}\nu(t)}\cong_\g M_{\lambda_\8}^{\sigma,\cai}$, the restriction of $\tau_{\nu(t)}(v_\8, \cdot,\dots,\cdot,\cdot, \wac_{w_1},\dots,\wac_{w_m})$ to  $\bigotimes_{i=1}^N \widetilde W_{\iota_{t-z_i} \nu(t)} \ox \widetilde W^\Gamma_{\iota_t \nu(t)}\cong_\g \bigotimes_{i=1}^N M_{\lambda_i}^*\ox M_{\lambda_0}^{*,\sigma}$ is proportional to the tensor product $v_{\lambda_1} \ox \dots v_{\lambda_N} \ox v_{\lambda_0}$ of the generating vectors $v_{\lambda_i}\in M_{\lambda_i}$, $i=1,\dots,N$, $v_{\lambda_0}\in M_{\lambda_0}^\sigma$. 
Now, it follows from \eqref{GEm} that $v_{\lambda_i}(E_\alpha \cdot) = -v_{\lambda_i}(G_\alpha \cdot)$. 
Because the two actions of $U(\n)$ commute, we therefore have 
\begin{multline} v_\lambda(G_{\alpha^{(1)}}G_{\alpha^{(2)}} \dots G_{\alpha^{(k)}} \cdot) = - v_\lambda(E_{\alpha^{(1)}}G_{\alpha^{(2)}} \dots G_{\alpha^{(k)}} \cdot)\\  
= -v_\lambda(G_{\alpha^{(2)}} \dots G_{\alpha^{(k)}} E_{\alpha^{(1)}} \cdot) = \dots = (-1)^k v_\lambda( E_{\alpha^{(k)}} \dots E_{\alpha^{(2)}} E_{\alpha^{(1)}} \cdot)\nn\end{multline} for any roots $\alpha^{(1)},\dots,\alpha^{(k)}\in \Delta^+$. 
Hence, by definition of the contragredient dual,
\be v_\lambda(G_{\alpha^{(1)}}G_{\alpha^{(2)}} \dots G_{\alpha^{(k)}} \cdot) 
  = (-1)^k(F_{\alpha^{(1)}} F_{\alpha^{(2)}} \dots F_{\alpha^{(k)}} v_\lambda)(\cdot).\ee

It follows from the definition of the two left commuting actions of $U(\n)$ on $M_{\lambda_i}^*$ that $v_{\lambda_i}(E_\alpha \cdot) = -v_{\lambda_i}(G_\alpha \cdot)$.\footnote{To see this consider, equivalently, the two commuting left actions of $U(\mf n_-)$ on the Verma module $M_{\lambda} \cong_\C U(\mf n_-)\ox_\C \C v_\lambda$. They are given by $X \on (n \ox v_\lambda) = Xn \ox v_\lambda$ and $X\triangleright (n\ox v_\lambda) = - nX \ox v_{\lambda}$, for $X\in \n_-$, $n\in U(\n_-)$. So in particular $X \on (1\ox v_\lambda) = - X\triangleright (1\ox v_\lambda)$.} Because these two actions commute, we therefore have 
\begin{multline} v_\lambda(G_{\alpha^{(1)}}G_{\alpha^{(2)}} \dots G_{\alpha^{(k)}} \cdot) = - v_\lambda(E_{\alpha^{(1)}}G_{\alpha^{(2)}} \dots G_{\alpha^{(k)}} \cdot)\\  
= -v_\lambda(G_{\alpha^{(2)}} \dots G_{\alpha^{(k)}} E_{\alpha^{(1)}} \cdot) = \dots = (-1)^k v_\lambda( E_{\alpha^{(k)}} \dots E_{\alpha^{(2)}} E_{\alpha^{(1)}} \cdot)\nn\end{multline} for any roots $\alpha^{(1)},\dots,\alpha^{(k)}\in \Delta^+$. 
Hence, by definition of the contragredient dual,
\be v_\lambda(G_{\alpha^{(1)}}G_{\alpha^{(2)}} \dots G_{\alpha^{(k)}} \cdot) 
  = (-1)^k(F_{\alpha^{(1)}} F_{\alpha^{(2)}} \dots F_{\alpha^{(k)}} v_\lambda)(\cdot).\ee
For the module at the origin the argument is the same, with $\n^\sigma$ in place of $\n$, by virtue of the following lemma.

\begin{lem}  $(E_{(0,\alpha)}[-1]\wac)^{W}_{(0)} = - (G_{(0,\alpha)}[-1]\wac)^{W}_{(0)} + \dots$ where $\dots$ are terms containing at least one factor $a^*_{(0,\gamma)}$, $\gamma\in \Delta^+_0$.
\end{lem}
\begin{proof}
We have the quasi-module map
\be Y_W(A,u) =: \sum_{n\in \Z} A^{W}_{(n)} u^{-n-1}.\nn\ee
Recall Lemma 5.4 from \cite{VY2}. Let $A$ be a state of definite degree $\deg A$. Then $R_{\omega^k} A := \omega^{k \deg A} \sigma^k A$. Hence
\be Y_W(A,u) = \frac 1 T \sum_{k=0}^{T-1} Y_W(R_{\omega^k} A, \omega^k u) 
 = \frac 1 T \sum_{k=0}^{T-1} \omega^{k\deg A} (\sigma^k A)^{W}_{(n)} (\omega^k u)^{-n-1}\ee
and thus
\be A^{W}_{(n)} = \frac 1 T \sum_{k=0}^{T-1} \left(\omega^{k (\deg A -n - 1)} \sigma^k A\right)^{W}_{(n)} = (\Pi_{n+1-\deg A} A)^{W}_{(n)} \ee

Now $G_{\alpha_i}[-1]\wac$ is a state of degree 1. Hence
\be (G_{\alpha_i}[-1]\wac)^{W}_{(0)} = ( \Pi_0 G_{\alpha_i}[-1] \wac)^{W}_{(0)} .\ee
That is, in terms of the adapted coordinates,
\be (G_{(j,\alpha)}[-1]\wac)^{W}_{(0)} = 0 \quad\text{for all}\,\, \alpha\in \Delta^+_j,  \text{ for every } j\in \Z_T\setminus\{ 0\}.\ee

Arguing as for $E_{(0,\alpha)}[-1]\wac$ in the proof of Proposition 4.4. of \cite{VY1}, we have
\be (G_{(0,\alpha)}[-1]\wac)^{W}_{(0)} = \sum_{\beta\in \Delta^+_0} R_{(0,\alpha)}^{(0,\beta)}(Ta^*_{(0,\bullet)}[0]) a_{(0,\beta)}[0] 
 =  a_{(0,\alpha)}[0] + \dots\ee
where $\dots$ are terms containing at least one factor $a^*_{(0,\gamma)}$, $\gamma\in \Delta^+_0$. At the same time, 
\be (E_{(0,\alpha)}[-1]\wac)^{W}_{(0)} = \sum_{\beta\in \Delta^+_0} P_{(0,\alpha)}^{(0,\beta)}(Ta^*_{(0,\bullet)}[0]) a_{(0,\beta)}[0] 
 =  -a_{(0,\alpha)}[0] + \dots\ee
so that $(E_{(0,\alpha)}[-1]\wac)^{W}_{(0)} = - (G_{(0,\alpha)}[-1]\wac)^{W}_{(0)} + \dots$.  
\end{proof}

This establishes that the weight function $\psi$ is indeed given recursively as in \eqref{psidef}.

\subsection{Proof of Theorem \ref{chi0thm}}
We keep the notations of the previous section. Let us now suppose that $\chi=0$. 
Then we have from \eqref{chitdef} that 
\begin{align} -\iota_{t^{-1}}\nu(t) &= 0+ \frac T t \left(\sum_{i=1}^N \Pi_0\lambda_i - \sum_{j=1}^m \Pi_0 \alpha_{c(j)}  + \lambda_0+\Lambda_0/T\right) + \mc O\left(\frac 1 {t^2}\right) \nn\\ 
&=
0+\frac T t \left(\lambda_\8 + \Lambda_0/T\right) + \mc O\left(\frac 1 {t^2}\right) \nn\end{align}
where $\lambda_\8$ is as in \eqref{l8def}. 

Recall that $\widetilde W^{\Gamma,\vee}_\nu \subset W^{\Gamma,\vee}_\nu$ denotes the subspace spanned by vectors of the form $\lwac a_{(0,\alpha_1)}[0] \dots a_{(0,\alpha_k)}[0]$, $k\in \Z_{\geq 0}$, $\alpha_1,\dots,\alpha_k\in \Delta_0^+$. This subspace has the structure of a right module over $U((\g\ox \C[[t^{-1}]])^\Gamma)$ on which $U((\g\ox t^{-1} \C[[t^{-1}]])^\Gamma)$ acts as zero. In particular it is a right module over $U(\g^\sigma)$.
Whenever, $t^2 \nu_\8$ has at most a simple pole in $t^{-1}$, as it does here, there is an isomorphism of right $U(\g^\sigma)$ modules
\be\widetilde W^{\Gamma,\vee}_{\nu_\8} \cong \Homres_\C(M^{*,\sigma}_{\frac1T(\res_{t^{-1}}(t^2\nu_\8)-\Lambda_0)},\C) \cong M^{\sigma,\cai}_{\frac1T(\res_{t^{-1}}(t^2\nu_\8)-\Lambda_0)},\label{wmig2}\ee
cf. \eqref{wmig}.

Now we modify the construction of \S\ref{sec: indg} as follows. Let $\M_\8$ be a module over $(\g^\op\otimes \C[[t^{-1}]])^{\Gamma}$ (rather than $(\g^\op\otimes t^{-1}\C[[t^{-1}]])^{\Gamma}$ as in \S\ref{sec: indg}). We make it into a module  $\M_\8^{k/T}$ over $(\g^\op\otimes\C[[t^{-1}]])^{\Gamma} \oplus \C K_\8$ by declaring that $K_\8$ acts by multiplication by $k/T\in \C$. We have the induced right module $U(\gh^{\Gamma}_{\8})$ module of level $k/T$,
\be
\MM_{\8}^{k/T} :=   \M_{\8}^{k/T}\otimes_{U((\g\otimes\C[[t^{-1}]])^\Gamma \oplus \C K_\8)} U(\gh^\Gamma_{\8}) .
\ee
The tensor product
\be \MMx := \MM^{k/T}_\8 \otimes \bigotimes_{i=1}^N \MM^k_{z_i} \otimes \MM^{k/T}_0 \label{MMxdef2}\ee
is again a module over $\gN$ on which $K$ acts as $k$. Pulling back by the embedding \eqref{iotamap}, we have that $\MMx$ becomes a module over $\gG$ and we can form the space of coinvariants $\MMx \big/ \gG$.
Let us write 
\be \gN^+:=(\g^\op \otimes \C[[t^{-1}]])^{\Gamma} \oplus \bigoplus_{i=1}^N \g\otimes\C[[t-z_i]] \oplus (\g \otimes \C[[t]])^{\Gamma}.\label{gNpdef2}\ee 
We have the natural inclusion $\g^\sigma \into \g$ and hence the embedding 
\be \g^\sigma\longinto 
\gN^+; \quad X\longmapsto (- X[0]_\8; X[0]_{z_1},\dots, X[0]_{z_N};X[0]_0) .\label{ed}\ee
Pulling back by this embedding, the tensor product 
\be \Mx := \M_\8 \otimes \bigotimes_{i=1}^N \M_{z_i} \otimes \M_0 \label{Mxdef2}\ee
is a module over $\g^\sigma$ and we have the space of coinvariants $\Mx\big/ \g^\sigma$.
\begin{prop}\label{prop:isom2} There is a canonical isomorphism of vector spaces
\be \MMx \big/ \gG \cong_\C \Mx\big/ \g^\sigma.\nn\ee
\end{prop}
\begin{proof}
We have
\be \MMx = U(\gN) \otimes_{U(\gN^+)} \Mx.\nn\ee 
Here we regard $\M_\8$ as a left module over $U(\g^{\sigma,\op})$.

There are natural embeddings of Lie algebras $\g^\sigma \into \gN^+$ and $\g^\sigma \into \gG$; the first is given in \eqref{ed}; for the second, $X\in \g^\sigma$ embeds as the constant function $X(t)=X$ in $\gG$. In turn, both  $\gN^+$ and $\gG$ embed into $\gN$ -- see \eqref{iotamap} --  and the following diagram of embeddings commutes
\be\begin{tikzpicture}    
\matrix (m) [matrix of math nodes, row sep=2em,    
column sep=4em, text height=1ex, text depth=1ex]    
{     
\g^\sigma & \gN^+  \\    
\gG & \gN . \\    
};    
\path[right hook->,font=\scriptsize,shorten <= 2mm,shorten >= 2mm]    
(m-1-1) edge node [above] {} (m-1-2)    
(m-2-1) edge node [above] {} (m-2-2);    
\path[right hook->,shorten <= 0mm,shorten >= 1mm]    
(m-1-1) edge node [left] {} (m-2-1)    
(m-1-2) edge node [right] {} (m-2-2);    
\end{tikzpicture}\nn\ee    
We may identify $\g^\sigma$, $\gN^+$ and $\gG$ with their images in $\gN$. Then
\be \gN = \gN^+ + \gG\quad\text{and}\quad \gN^+ \cap \gG = \g^\sigma.\nn\ee
Therefore \cite[Proposition 2.2.9]{Dixmier} there is an isomorphism
\be U(\gN) \cong U(\gG) \otimes_{U(\g^\sigma)} U(\gN^+)\nn\ee
of vector spaces and in fact of left $U(\gG)$-modules. Hence we have an isomorphism of left $U(\gG)$-modules,
\be
\MMx \cong 
  U(\gG) \otimes_{U(\g^\sigma)} U(\gN^+) \otimes_{U(\gN^+)} \Mx
 = U(\gG) \otimes_{U(\g^\sigma)} \Mx.\nn
\ee
The result follows by the following elementary lemma.
\end{proof}
\begin{lem}\label{lem: cancelling} Suppose $\mf b \into \mf a$ is an embedding of complex Lie algebras. Let $V$ be a $\mf b$-module and $M$ an $\mf a$-module. Then $((U(\mf a) \otimes_{U(\mf b)} V)\otimes M)\big/ \mf a \cong_\C (V\otimes M)\big/ \mf b$.

In particular (taking $M$ to be the trivial one-dimensional module), $(U(\mf a) \otimes_{U(\mf b)} V)\big/ \mf a \cong_\C V\big/ \mf b$.
\end{lem}
\begin{proof}
Consider the linear map $V\otimes M \cong_\C \left( U(\mf b) \otimes_{U(\mf b)} V\right) \otimes M \into  \left( U(\mf a) \otimes_{U(\mf b)} V\right) \otimes M \onto \left(\left( U(\mf a) \otimes_{U(\mf b)} V\right) \otimes M\right)\big/\mf a$ sending $v\otimes m \mapsto [(1 \otimes v) \otimes m]$. 
This map is surjective.
We must show that it has kernel $\mf b \on ( V\otimes M)$. And indeed, the kernel is the intersection of $\left( U(\mf b) \otimes_{U(\mf b)} V\right) \otimes M$ with $\mf a \on \left( \left(U(\mf a) \otimes_{U(\mf b)} V\right)\otimes M\right)$, which is $\mf b \on \left( \left(U(\mf b) \otimes_{U(\mf b)} V\right)\otimes M\right) \cong_\C \mf b \on \left(V\otimes M\right)$. 
\end{proof}

We may now choose modules -- cf. \eqref{mmods} --
\begin{alignat}{2} \M_{z_i} &= M_{\lambda_i}^* && \cong \widetilde W_{\res_{t-z_i}\nu(t)}, \qquad i=1,\dots,N, \nn\\
 \M_0 &= M^{*,\sigma}_{\lambda_0} && \cong \widetilde W^\Gamma_{\res_t \nu(t)}, \nn\\
  \M_\8 &= M^{\sigma,\cai}_{\lambda_\8}
          &&\cong \widetilde W^{\Gamma,\vee}_{-\res_{t^{-1}\nu(t)}}, \end{alignat}
where the isomorphisms are now as in \eqref{wmi}, \eqref{wmig}, \eqref{wmig2}. 
So in total we again get a canonical map of $\gh_{\8,p,0}$-modules
$\MM^{-h^\vee/T}_\8 \ox \bigotimes_{i=1}^N \MM^{-h^\vee}_{z_i} \ox \MM^{-h^\vee/T}_0 \to W^{\Gamma,\vee}_{-\iota_{t^{-1}}\nu(t)} \ox \bigotimes_{i=1}^N W_{\iota_{t-z_i} \nu(t)} \ox W^\Gamma_{\iota_t \nu(t)}$ and hence, from the $\g^\Gamma_\zsa$-invariant functional $\psi_{\nu(t)}$ of \eqref{psidef}, we obtain a $\g^\Gamma_\zsa$-invariant functional $\MM^{-h^\vee/T}_\8 \ox \bigotimes_{i=1}^N \MM^{-h^\vee}_{z_i} \ox \MM^{-h^\vee/T}_0 \to \C$. 
By Proposition \ref{prop:isom2} this is the same thing as a $\g^\sigma$-invariant functional 
\be M^{\sigma,\cai}_{\lambda_\8} \ox \bigotimes_{i=1}^N M_{\lambda_i}^*\ox M_{\lambda_0}^{*,\sigma}  \to \C.\label{lf2}\ee 
Note, moreover that the restriction of this functional to the subspace $\C v^{\sigma,\cai}_{\lambda_\8} \ox \bigotimes_{i=1}^N M_{\lambda_i}^*\ox M_{\lambda_0}^{*,\sigma} \cong \bigotimes_{i=1}^N M_{\lambda_i}^*\ox M_{\lambda_0}^{*,\sigma}$ coincides up to non-zero rescaling with the linear functional of \eqref{lf}. Indeed, this follows from the fact that there are in both cases zero contribution from ``swapping to infinity'', cf. \eqref{8swap}. 

The result then follows by the following proposition.

\begin{prop}\label{prop: hom} Given $\lambda_\8,\lambda_0\in \h^{*,\sigma}$ and $\lambda_1,\dots,\lambda_N\in \h^*$, there is an isomorphism of vector spaces, and in fact of left $(U(\g)^{\ox N}\ox U(\g^\sigma))^{\g^\sigma}$-modules,
\be \left(\bigotimes_{i=1}^N M_{\lambda_i}\ox M_{\lambda_0}^\sigma\right)_{\lambda_\8}^{\n^\sigma} \cong 
\Hom_{\g^\sigma}\left( (M_{\lambda_\8}^{\sigma})^{\cai\circ S} \ox \bigotimes_{i=1}^NM_{\lambda_i}^*\ox M_{\lambda_0}^{*,\sigma}, \C\right)\ee
\end{prop}
\begin{proof}
We have
\begin{align} \left(\bigotimes_{i=1}^N M_{\lambda_i}\ox M_{\lambda_0}^\sigma\right)_{\lambda_\8}^{\n^\sigma} &=
\Hom_{\h^\sigma\oplus\n^\sigma}\left(\C v_{\lambda_\8}, \bigotimes_{i=1}^N M_{\lambda_i}\ox M_{\lambda_0}^\sigma\right)\nn\\
&\cong 
\Hom_{\g^\sigma}\left(U(\g^\sigma) \ox_{U(\h^\sigma\oplus\n^\sigma)} \C v_{\lambda_\8}, \bigotimes_{i=1}^N M_{\lambda_i}\ox M_{\lambda_0}^\sigma\right)\nn\\
&=
\Hom_{\g^\sigma}\left( M^\sigma_{\lambda_\8}, \bigotimes_{i=1}^NM_{\lambda_i}\ox M_{\lambda_0}^\sigma\right).\end{align}
This is the space of those maps $\phi: M^\sigma_{\lambda_\8}\to \bigotimes_{i=1}^NM_{\lambda_i}\ox M_{\lambda_0}^\sigma$ such that 
\be \phi(x\on v) = (\Delta^{N+1} x)\on \phi(v),\quad v\in M^\sigma_{\lambda_\8}, x\in \g^\sigma\label{ec}.\ee It carries a natural left action of $(U(\g)^{\ox N}\ox U(\g^\sigma))^{\g^\sigma}$: given one such a map $\phi$, the map $v\mapsto X\phi(v)$ is another, for any $X\in (U(\g)^{\ox N}\ox U(\g^\sigma))^{\g^\sigma}$.

Then 
\be \Hom_{\g^\sigma}\left( M^\sigma_{\lambda_\8}, \bigotimes_{i=1}^NM_{\lambda_i}\ox M_{\lambda_0}^\sigma\right)
\cong
\Hom_{\g^\sigma}\left( M_{\lambda_\8} \ox \left(\bigotimes_{i=1}^NM_{\lambda_i}\ox M_{\lambda_0}^{\sigma}\right)^{\vee,S}, \C\right)
\ee
where for a finitely-weighted left $U(\g^\sigma)$-module $M$ we denote by $M^{\vee,S}$ its restricted dual made into a left $U(\g^\sigma)$-module by twisting by $S$. Indeed, given a map $\phi$ such that \eqref{ec} holds we define a map $\Phi: M_{\lambda_\8} \ox  \left(\bigotimes_{i=1}^NM_{\lambda_i}\ox M_{\lambda_0}^{\sigma}\right)^{\vee,S}\to \C$ by $\Phi(v\ox \mu) = \mu(\phi(v))$, $v\in M_{\lambda_\8}$, $\mu\in \left(\bigotimes_{i=1}^NM_{\lambda_i}\ox M_{\lambda_0}^{\sigma}\right)^{\vee,S}$. It obeys $\Phi( x \on (v\ox \mu)) = 0$ where $x\on (v\ox \mu) := (x\on v)\ox \mu + x \ox (\mu \circ (-\Delta^Nx\on))$, for $x\in \g^\sigma$.

Finally, we may twist the action of $\g^\sigma$ on $M_{\lambda_\8} \ox \left(\bigotimes_{i=1}^NM_{\lambda_i}\ox M_{\lambda_0}^{\sigma}\right)^{\vee,S}$ by the automorphism $S\circ \cai=\cai\circ S$. Since $S^2=\id$ here,  $(M_\lambda^{\vee,S})^{S\circ \cai} = M_\lambda^{\vee,\cai}=: M_\lambda^*$. Thus
\be
\Hom_{\g^\sigma}\left( M_{\lambda_\8}^{\sigma} \ox \left(\bigotimes_{i=1}^NM_{\lambda_i}\ox M_{\lambda_0}^{\sigma}\right)^{\vee,S}, \C\right)
\cong
\Hom_{\g^\sigma}\left( (M_{\lambda_\8}^{\sigma})^{\cai\circ S} \ox \bigotimes_{i=1}^NM_{\lambda_i}^*\ox M_{\lambda_0}^{*,\sigma}, \C\right)
\ee
and this is what we mean by $\Hom_\C\left( \left(M_{\lambda_\8}^{\sigma,\cai} \ox \bigotimes_{i=1}^N M_{\lambda_i}^*\ox M_{\lambda_0}^{*,\sigma}\right)\big/ \g^\sigma, \C\right)$. (Recall that we embed $\g^\sigma \into \g^{\sigma,\op} \oplus \bigoplus_{i=1}^N \g^\sigma$, $X\mapsto (-X,X,\dots,X)$ in the definition of the coinvariants here.)
\end{proof}

\appendix

\section{Proof of Proposition \ref{qHprop}}\label{qHapp}
Let us write 
\be \dell n x := \frac{1}{n!} \left(\frac{\del}{\del x}\right)^n \nn\ee
Consider the $\Gamma$-equivariant rational function 
\be \label{swap function}
f(t) = \sum_{k = 0}^{T-1} \frac{\sigma^k A}{(\omega^{-k} t - u)^{p+1}} = \dell pu \sum_{k = 0}^{T-1}\frac{\sigma^k A}{\omega^{-k} t - u}  \in \gGu.
\ee
The expansion of \eqref{swap function} at $u$ is
\be \iota_{t-u}f(t) = A[-p-1]_u - \dell p u \sum_{k=1}^{T-1} \sum_{n=0}^{\infty} \frac{\omega^{k (p+1)}}{(\omega^k - 1)^{n+p+1} u^{n+1}} (\sigma^k A)[n]_u \in \g\otimes\C((t-u)). \ee
Its expansions at $z_i$, $i=1,\dots,N$, and at $0$ are given by
\be \iota_{t-z_i}f(t) = - \dell p u \sum_{k=0}^{T-1} \sum_{n=0}^{\infty} \frac{\omega^{-kn}}{(u - \omega^{-k}z_i)^{n+1}} (\sigma^k A)[n]_{z_i} \in \g\otimes\C[[t-z_i]], \ee
\be \iota_{t}f(t) = - \dell p u \sum_{n=0}^{\infty} \frac{1}{u^{n+1}} (\Pi_n A)[n]_0 \in \g\otimes\C[[t]], \ee
where, recall, $A[n]_{z_i} = A \otimes  (t - z_i)^n \in \g\otimes\C((t-z_i))$. 
It is regular at $\8$ and its expansion there is given by
\be - \iota_{t^{-1}}f(t) =
 - \dell pu \sum_{k=0}^{T} \sum_{n=0}^\8 (\omega^k u)^n \omega^k \sigma^k A[-n-1]_\8
= -\dell pu \sum_{n=0}^\8 u^n (\Pi_{-n-1} A)[-n-1]_\8, \ee 
where recall that for the modes at infinity we write $A[n]_\8 = A \ox t^n$ in our conventions. 

Thus we have in particular
\begin{multline} \label{quad swap}
\big[ S(u) \cdot x \otimes \vac \big] :=
\big[ x \otimes \half I^a[-1] I_a[-1] \vac \big] = \left[ I^a(u) \cdot x \otimes \ha I_a[-1] \vac \right]\\
+ \sum_{p=1}^{T-1} \frac{\omega^p}{(\omega^p - 1) u} \big[ x \otimes \ha (\sigma^p I^a)[0] I_a[-1] \vac \big]
+ \sum_{p=1}^{T-1} \frac{\omega^p}{(\omega^p - 1)^2 u^2} \big[ x \otimes \ha (\sigma^p I^a)[1] I_a[-1] \vac \big]\\
\end{multline}
where, putting the expansions at $0, z_1,\dots,z_N,\8$ together, we shall write, for any $A\in \g$,
\be A(u) := \left(  \sum_{i=1}^N \sum_{n=0}^\8 \sum_{k=0}^{T-1} \frac{\omega^{-kn}(\sigma^k A)[n]_{z_i} }{(u - \omega^{-k}z_i)^{n+1}} + T \sum_{n=0}^{\infty} \frac{ (\Pi_n A)[n]_0}{u^{n+1}} + T \sum_{n=0}^\8 u^n (\Pi_{-n-1} A)[-n-1]_\8 \right).\ee
Then, defining the element $F\in \g^\sigma$ and number $K$ as in \eqref{FKdef}, 
we see after one further ``swapping'' that
\be S(u) = \half I_a(u) I^a(u) + \frac 1u F(u) + \frac 1 {u^2} K.\label{Surat}\ee
Consider first the terms in $\half I_a(u) I^a(u)$ acting at sites $i,j\in \{1,\dots,N\}$, $i\neq j$. Let us write 
\be z_{ik} := \omega^{-k} z_i\quad\text{and}\quad \dell{n}{ik} = \frac{1}{n!} \left(\frac{\del}{\del z_{ik}}\right)^n.\ee
Then $\frac 1 {(u-z_{ik})^{n+1}} = \dell n{ik} \frac1 {u-z_{ik}}$ and we get the terms
\begin{multline} \half \sum_{n,m=0}^\8 \sum_{\substack{i,j=1\\j\neq i}}^N \sum_{k,l=0}^{T-1} 
\dell n {ik} \dell m {jl} \frac 1 {u-z_{ik}} \frac 1 {u-z_{jl}} \omega^{-kn-lm} (\sigma^k I_a)[n]_{z_i} (\sigma^l I^a)[m]_{z_j} \\
 = \half \sum_{n,m=0}^\8 \sum_{\substack{i,j=1\\j\neq i}}^N \sum_{k,l=0}^{T-1} 
\dell n {ik} \dell m {jl} \left( \frac 1 {u-z_{ik}} \frac 1 {z_{ik}-z_{jl}} + \frac 1 {u-z_{jl}} \frac 1 {z_{jl}-z_{ik}} \right) \omega^{-kn-lm} (\sigma^k I_a)[n]_{z_i} (\sigma^l I^a)[m]_{z_j}\\
= \half \sum_{n,m=0}^\8 \sum_{\substack{i,j=1\\j\neq i}}^N \sum_{k,l=0}^{T-1} 
\dell n {ik} \dell m {jl} \frac 1 {u-z_{ik}} \frac 1 {z_{ik}-z_{jl}}  \omega^{-kn-lm} \left\{ (\sigma^k I_a)[n]_{z_i} , (\sigma^l I^a)[m]_{z_j} \right\}\\
= \sum_{n,m=0}^\8 \sum_{\substack{i,j=1\\j\neq i}}^N \sum_{k,l=0}^{T-1} 
\dell n {ik} \dell m {jl} \frac 1 {u-z_{ik}} \frac 1 {z_{ik}-z_{jl}}  \omega^{-kn-lm} (\sigma^k I_a)[n]_{z_i} (\sigma^l I^a)[m]_{z_j} 
\end{multline}
where $\{a,b\} := ab+ba$ denotes the anti-commutator. Now here
\begin{multline} \dell n {ik} \dell m {jl} \frac 1 {u-z_{ik}} \frac 1 {z_{ik}-z_{jl}} 
 = \sum_{p=0}^n \left( \dell p {ik} \frac 1 {u-z_{ik}} \right)
    \left( \dell {n-p} {ik} \dell m {jl} \frac 1{z_{ik} - z_{jl}} \right) \\
= \sum_{p=0}^n \frac 1 {(u-z_{ik})^{p+1}} (-1)^{n-p} \binom {n+m-p} m  \frac 1{(z_{ik} - z_{jl})^{n-p+m+1}}.
\end{multline}
Thus we find the following terms in  $\half I_a(u) I^a(u)$:
\begin{multline}   
\sum_{n,m=0}^\8 \sum_{\substack{i,j=1\\j\neq i}}^N \sum_{k,l=0}^{T-1} \sum_{p=0}^n 
\frac{  \omega^{-kn-lm}}{(u-z_{ik})^{p+1}} (-1)^{n-p} \binom{n+m-p}{m}  \frac 1{(z_{ik} - z_{jl})^{n-p+m+1} }(\sigma^k I_a)[n]_{z_i} (\sigma^l I^a)[m]_{z_j} \\
=
\sum_{i=1}^N \sum_{k=0}^{T-1} \sum_{p=0}^\8 \frac 1 {(u-\omega^{-k} z_i)^{p+1}}
\left( \sum_{\substack{j=1\\j\neq i}}^N \sum_{l=0}^{T-1} \sum_{r,m = 0}^\8  \frac{\omega^{ - k(r+p) - lm} (-1)^r \binom{r+m} m}{(\omega^{-k} z_i - \omega^{-l} z_j)^{r+m+1}} 
(\sigma^k I_a)[r+p]_{z_i} (\sigma^l I^a)[m]_{z_j}\right)\\
=
\sum_{i=1}^N \sum_{k=0}^{T-1} \sum_{p=0}^\8 \frac{\omega^{-kp + k }\sigma^k} {(u-\omega^{-k} z_i)^{p+1}}
\left( \sum_{\substack{j=1\\j\neq i}}^N \sum_{l=0}^{T-1} \sum_{r,m = 0}^\8  \frac{\omega^{(k-l)m} (-1)^r \binom{r+m} m}{(z_i - \omega^{k-l} z_j)^{r+m+1}} 
I_a[r+p]_{z_i} (\sigma^{l-k} I^a)[m]_{z_j}\right)\\
=
\sum_{i=1}^N \sum_{k=0}^{T-1} \sum_{p=0}^\8 \frac{\omega^{-kp + k }\sigma^k} {(u-\omega^{-k} z_i)^{p+1}}
\left( \sum_{\substack{j=1\\j\neq i}}^N \sum_{l=0}^{T-1} \sum_{r,m = 0}^\8  \frac{\omega^{-lm} (-1)^r \binom{r+m} m}{(z_i - \omega^{-l} z_j)^{r+m+1}} 
I_a[r+p]_{z_i} (\sigma^{l} I^a)[m]_{z_j}\right)
\end{multline}
In part by a similar calculation we find the terms in  $\half I_a(u) I^a(u)$ acting solely at one site $i\in \{1,\dots, N\}$, namely
\begin{multline}
\sum_{i=1}^N \sum_{k=0}^{T-1} \sum_{p=0}^\8 \frac{\omega^{-kp + k }\sigma^k} {(u-\omega^{-k} z_i)^{p+1}} \times \\
\left( \sum_{l=1}^{T-1} \sum_{r,m = 0}^\8  \frac{\omega^{-lm} (-1)^r \binom{r+m} m}{((1  - \omega^{-l}) z_i)^{r+m+1}} 
\half \left\{I_a[r+p]_{z_i}, (\sigma^{l} I^a)[m]_{z_i}\right\} + \sum_{n=0}^{p-1} \half I_a[n]_{z_i} I^a[p-n-1]_{z_i}\right),
\end{multline}
where the second term is an ``on-diagonal'' term in $\half I_a(u) I^a(u)$.

The cross terms in $\half I_a(u) I^a(u)$ acting at sites $i\in \{1,\dots,N\}$ and $0$ are given by (here we use an obvious trick to make the calculation resemble the one above)
\begin{multline} T \sum_{n,m=0}^\8 \sum_{i=1}^N \sum_{k=0}^{T-1} 
\left.\dell n {ik} \dell m {z} \frac 1 {u-z_{ik}} \frac 1 {u-z} \right|_{z=0}\omega^{-kn} (\sigma^k I_a)[n]_{z_i} (\Pi_m I^a)[m]_0 \\
 = T\sum_{n,m=0}^\8 \sum_{i=1}^N \sum_{k=0}^{T-1} 
\left.\dell n {ik} \dell m {z} \left( \frac 1 {u-z_{ik}} \frac 1 {z_{ik}-z} + \frac 1 {u-z} \frac 1 {z-z_{ik}} \right)\right|_{z=0} \omega^{-kn} (\sigma^k I_a)[n]_{z_i} (\Pi_m I^a)[m]_0\\
= T\sum_{n,m=0}^\8 \sum_{i=1}^N \sum_{k=0}^{T-1} 
\left. 
\sum_{p=0}^n \frac 1 {(u-z_{ik})^{p+1}}  \frac {(-1)^{n-p} \binom {n+m-p} m }{(z_{ik} - z)^{n-p+m+1}}
\right|_{z=0} \omega^{-kn} (\sigma^k I_a)[n]_{z_i} (\Pi_m I^a)[m]_0\\
+ T\sum_{n,m=0}^\8 \sum_{i=1}^N \sum_{k=0}^{T-1} 
\left.
\sum_{p=0}^n \frac 1 {(u-z)^{p+1}}  \frac{(-1)^{m-p} \binom {n+m-p} m }{(z-z_{ik})^{n-p+m+1}}
\right|_{z=0} \omega^{-kn} (\sigma^k I_a)[n]_{z_i} (\Pi_m I^a)[m]_0\\
= T\sum_{n,m=0}^\8 \sum_{i=1}^N \sum_{k=0}^{T-1}  
\sum_{p=0}^n \frac 1 {(u-\omega^{-k}z_{i})^{p+1}}   \frac {(-1)^{n-p} \binom {n+m-p} m}{z_{i}^{n-p+m+1}}
 \omega^{k(-p+m+1)} (\sigma^k I_a)[n]_{z_i} (\Pi_m I^a)[m]_0\\
+ T\sum_{n,m=0}^\8 \sum_{i=1}^N \sum_{k=0}^{T-1} 
\sum_{p=0}^n \frac 1 {u^{p+1}}  \frac{(-1)^{n+1} \binom {n+m-p} m }{z_{i}^{n-p+m+1}}
\omega^{k(-p+m+1)} (\sigma^k I_a)[n]_{z_i} (\Pi_m I^a)[m]_0\\
\end{multline}
Next, the cross terms in $\half I_a(u) I^a(u)$ acting at sites $i\in \{1,\dots,N\}$ and $\8$ are
\begin{multline} 
T \sum_{k=0}^{T-1} \sum_{n,m = 0}^\8 \sum_{i=1}^N \frac{\omega^{-kn}}{(u-\omega^{-k} z_i)^{n+1}} u^m (\Pi_{-m-1} I_a)[-m-1]_\8 (\sigma^k I^a)[n]_{z_i}\\
= T \sum_{k=0}^{T-1} \sum_{n,m = 0}^\8 \sum_{i=1}^N 
\left( \sum_{p=0}^{m-n-1} \binom {n+p} n u^{m-n-1-p} (\omega^{-k}z_i)^p  
   + \sum_{p=0}^n \binom mp \frac{(\omega^{-k}z_i)^{m-p}}{(u-\omega^{-k}z_i)^{n-p+1}}\right) \times \\ \omega^{-kn}(\Pi_{-m-1} I_a)[-m-1]_\8 (\sigma^k I^a)[n]_{z_i} \label{8i}
\end{multline}
where we used the identity
\be \frac{u^m}{(u-z)^{n+1}} = \sum_{p=0}^{m-n-1} \binom {n+p} n u^{m-n-1-p} z^p + \sum_{p=0}^n \binom mp \frac{z^{m-p}}{(u-z)^{n-p+1}}.\nn\ee 
The pole term in $(u-\omega^{-k} z_i)$ in \eqref{8i} is (here $r=n-p$, $n=r+p$)
\begin{multline}
T \sum_{i=1}^N \sum_{k=0}^{T-1} \sum_{r=0}^\8 \frac 1{(u- \omega^{-k} z_i)^{r+1}} \sum_{p,m=0}^\8 \omega^{-kr-km} z_i^{m-p} \binom m p  (\Pi_{-m-1} I_a)[-m-1]_\8 (\sigma^k I^a)[r+p]_{z_i} \\
= T \sum_{i=1}^N \sum_{k=0}^{T-1} \sum_{r=0}^\8 \frac {\omega^{-kr+k}\sigma^k}{(u- \omega^{-k} z_i)^{r+1}}\left( 
\sum_{p,m=0}^\8  z_i^{m-p} \binom m p  (\Pi_{-m-1} I_a)[-m-1]_\8 I^a[r+p]_{z_i} \right)
\\
= T \sum_{i=1}^N \sum_{k=0}^{T-1} \sum_{r=0}^\8 \frac {\omega^{-kr+k}\sigma^k}{(u- \omega^{-k} z_i)^{r+1}}\left( 
\sum_{n,p=0}^\8  z_i^{n} \binom {n+p} p  (\Pi_{-n-p-1} I_a)[-n-p-1]_\8 I^a[r+p]_{z_i} \right)
\end{multline}
where we use the fact that $\sigma^k \Pi_{-m-1} = \omega^{-km-k} \Pi_{-m-1}$.
The power of $u$ in \eqref{8i} is (here $r=m-n-1-p$, $p=m-n-1-r$)
\be T \sum_{r=0}^\8 u^r \sum_{i=1}^N \sum_{k=0}^{T-1} \sum_{m,n=0}^\8 
\omega^{-km+k+kr} z_i^{m-n-1-r} \binom {m-1-r} n 
(\Pi_{-m-1} I_a)[-m-1]_\8 (\sigma^k I^a)[n]_{z_i} . 
\ee

The cross terms in $\half I_a(u) I^a(u)$ acting at sites $0$ and $\8$ can be conveniently written as
\begin{multline}
  T^2 \sum_{r=0}^\8 \frac 1{u^{r+1}} \sum_{n=0}^r (\Pi_n I_a)[n]^{(0)} (\Pi_{r-n+1} I^a)[-n+r-1]_\8 \\
+ T^2 \sum_{r=0}^\8 u^r  \sum_{n=0}^\8 (\Pi_n I_a)[n]_0 (\Pi_{-r-n-2} I^a)[-n-r-2]_\8. 
\end{multline}

Finally the diagonal terms in  $\half I_a(u) I^a(u)$ acting at sites $0$ and $\8$ respectively are
\be \frac{T^2}{2} \sum_{r=1}^\8 \frac 1{u^{r+1}} \sum_{n=0}^{r-1} (\Pi_n I_a)[n]_0 (\Pi_{r-n-1} I_a)[r-n-1]_0
\ee
and
\be \frac{T^2}{2} \sum_{r=0}^\8 u^r \sum_{n=0}^r (\Pi_{-n-1} I_a)[-n-1]_\8 (\Pi_{-r+n-1} I_a)[-r+n-1]_\8.
\ee

Next we turn to the term $F(u)/u$ in \eqref{Surat}. Noting that
\be \frac 1 u \frac 1 {(u-z)^{n+1}} = \sum_{p=0}^n (-1)^p \frac{1}{z^{p+1}} \frac{1}{(u-z)^{n-p+1}} - \frac 1 u \frac {(-1)^n}{z^{n+1}}.\nn\ee
we find that in $F(u)/u$ the terms acting at the sites $i\in \{1,\dots,N\}$ are
\begin{multline} \sum_{i=1}^N\sum_{k=0}^{T-1} \sum_{n=0}^\8 \left( \sum_{p=0}^n (-1)^p \frac{\omega^{-kn}}{ (\omega^{-k} z_i)^{p+1} (u-\omega^{-k} z_i)^{n-p+1}}
- \frac 1 u  \frac{\omega^{-kn} (-1)^n}{(\omega^{-k} z_i)^{n+1}} \right)\sigma^k F[n]_{z_i} \\
=
\sum_{i=1}^N\sum_{k=0}^{T-1} \sum_{r=0}^\8 \sum_{n=r}^\8  (-1)^{n-r} \frac{\omega^{-kn+k(n-r+1)}}{(u-\omega^{-k} z_i)^{r+1}}
\frac{1}{z_i^{n-r+1}} \sigma^k F[n]_{z_i}
 -  \frac 1 u \sum_{i=1}^N \sum_{n=0}^\8  \frac {(-1)^n}{z_i^{n+1}} \sum_{k=0}^{T-1}\omega^{k} \sigma^k F[n]_{z_i}
\\
=
\sum_{i=1}^N\sum_{k=0}^{T-1} \sum_{r=0}^\8 \sum_{n=0}^\8  (-1)^{n} \frac{\omega^{-kr+k}}{(u-\omega^{-k} z_i)^{r+1}}
\frac{1}{z_i^{n+1}} \sigma^k F[n+r]_{z_i}
 -  \frac T u \sum_{i=1}^N \sum_{n=0}^\8  \frac {(-1)^n}{z_i^{n+1}} \Pi_{-1} F[n]_{z_i}.
\end{multline}
The terms acting at the sites 0 and $\8$ in $F(u)/u$ are clearly
\be T\sum_{n=0}^\8 \frac{1}{u^{n+2}} (\Pi_{n} F)[n]_0 + T \sum_{n=0}^\8 u^{n-1} (\Pi_{-n-1} F)[-n-1]_\8 \nn\ee
Now we collect terms, and use $\sigma^k \Pi_m = \omega^{km} \Pi_m$, $\sigma^k \mc H_{i,p} = \mc H_{i,p}$, and the fact that $\Pi_kF = 0$ for all $k\in \ZT\setminus \{0\}$, to obtain the result.

\section{$Y_W$-map} \label{app: YWmap}

In this appendix we use the notion of a vertex Lie algebra and related concepts as defined in \cite{VY2}. In particular, all our vertex Lie algebras are finitely generated as $\C[\D]$-modules, of the form $\C[\D] \otimes (L^o \oplus \C \cent) \big/ \C[\D] \otimes \C \cent$, where $L^o$ is a finite dimensional vector space and $\D$ is the translation operator of the vertex Lie algebra. To any such vertex Lie algebra $\vla$ we can associate a genuine Lie algebra, denoted $\U(\vla)$, which is isomorphic as a vector space to $L^o \otimes \C((t)) \oplus \C \cent$. The element in $\U(\vla)$ corresponding to $a \otimes t^n$ with $a \in L^o$ and $n \in \Z$ is denoted $a(n) \in \U(\vla)$. We also have a vector space direct sum decomposition $\U(\vla) = \U^-(\vla) \dotplus \U^+(\vla)$ where $\U^\pm(\vla)$ are Lie subalgebras isomorphic as vector spaces to $L^o \otimes \C[[t]] \oplus \C \cent$ and $L^o \otimes t^{-1} \C[t^{-1}]$ respectively. Denote by $\ueva$ the vacuum Verma module over $\U(\vla)$, namely the $\U(\vla)$-module induced from the one-dimensional $\U^+(\vla)$-module $\C \vac$ on which $\cent$ acts as $1$ and $L^o \otimes \C[[t]]$ acts trivially. Finally, let $\U(\vla)^\Gamma$ denote the subalgebra of $\Gamma$-invariant elements, where the action of $\Gamma$ on $\U(\vla)$ is defined by letting $\alpha \in \Gamma$ send any $a(n)$ with $a \in \vla$ and $n \in \Z$ to $\alpha^{-n-1} (R_\alpha a)(n)$. Given any $a \in \vla$ and $n \in \Z$ we define the corresponding \emph{twisted $n^{\rm th}$-mode}
\begin{equation*}
a^\Gamma(n) := \sum_{\alpha \in \Gamma} \alpha^{-n-1} (R_\alpha a)(n) \in \U(\vla)^\Gamma.
\end{equation*}

\subsection{`Local' and `global' Lie algebras}

Fix a set $\bm x := \{ x_i \}_{i=1}^p$ of $p \in \Z_{\geq 0}$ non-zero points in the complex plane. We attach to each point in $\bm x$ a local copy of the Lie algebra $\U(\vla)$, which we denote $\U(\vla)_{x_i}$. We will also use the index $x_i$ on the formal modes $a(n)$ when we wish to emphasise that these belong to $\U(\vla)_{x_i}$. Simiarly, to the origin we attach a local copy $\U(\vla)_0^\Gamma$ of the $\Gamma$-invariant subalgebra $\U(\vla)^\Gamma$ and to infinity we attach a local copy $\U(\vla^{\rm op})_\infty^\Gamma$ of the opposite Lie algebra $\U(\vla^{\rm op})^\Gamma$, where $\vla^{\rm op}$ denotes the vertex Lie algebra defined as the $\C[\D]$-module $\vla$ but with opposite $n^{\rm th}$-products. Consider their direct sum $\U(\vla^{\rm op})^\Gamma_{\infty} \oplus \bigoplus_{i=1}^p \U(\vla)_{x_i} \oplus \U(\vla)^\Gamma_0$. We define the ideal
\begin{equation} \label{Ip ideal}
I_{\infty, p, 0} := \text{span}_\C \big\{ \cent(-1)_{x_i} - T \cent(-1)_\infty \big\}_{i=1}^p \cup \big\{ \cent(-1)_0 - \cent(-1)_\infty \big\},
\end{equation}
and the corresponding quotient Lie algebra
\begin{equation*}
\U(\vla)_{\infty, \bm x, 0} := \U(\vla^{\rm op})^\Gamma_{\infty} \oplus \bigoplus_{i=1}^p \U(\vla)_{x_i} \oplus \U(\vla)^\Gamma_0 \bigg/ I_{\infty, p, 0}.
\end{equation*}

Let $\C_{\infty, \Gamma \bm x, 0}(t)$ be the algebra of global rational functions with poles at most at $0$, $\infty$ and the points in $\Gamma \bm x$. It comes equipped with the derivation $\partial_t$ and an action of $\Gamma$ defined through pullback by the multiplication map $t \mapsto \alpha^{-1} t$ for $\alpha \in \Gamma$. Consider the associated Lie algebra $\Lie_{\C_{\infty, \Gamma \bm x, 0}(t)} \vla$ defined with the help of \cite[Lemma 2.2]{VY2} for $\mathcal A = \C_{\infty, \Gamma \bm x, 0(t)}$. We will denote by $a f := \rho(a \otimes f)$ the class in $\Lie_{\C_{\infty, \Gamma \bm x, 0}(t)} \vla$ of an element $a \otimes f \in \vla \otimes \C_{\infty, \Gamma \bm x, 0}(t)$. The action of $\Gamma$ on $\vla \otimes \C_{\infty, \Gamma \bm x, 0}(t)$ defined for any $\alpha \in \Gamma$ by
\begin{equation*}
\alpha . (a \otimes f(t)) := \alpha^{-1} R_\alpha a \otimes f(\alpha^{-1} t),
\end{equation*}
gives rise to an action on $\Lie_{\C_{\infty, \Gamma \bm x, 0}(t)} \vla$ by Lie algebra automorphisms, \emph{cf.} \cite[Lemma 2.13]{VY2}.

Consider the ideal in $\Lie_{\C_{\infty, \Gamma \bm x, 0}(t)} \vla$ defined by
\begin{equation} \label{Ix ideal}
I_{\infty, \Gamma \bm x, 0} := \text{span}_{\C} \bigg\{ \sum_{\alpha \in \Gamma} \frac{\cent}{t - \alpha x} \bigg\}_{x \in \bm x \cup \{ 0 \}} = \text{span}_{\C} \bigg\{ \sum_{\alpha \in \Gamma} \frac{\cent}{t - \alpha x_i} \bigg\}_{i=1}^p \cup \left\{ \frac{T \cent}{t} \right\},
\end{equation}
and denote the corresponding quotient Lie algebra by
\begin{equation*}
\U_{\infty, \Gamma \bm x, 0}(\vla) := \Lie_{\C_{\infty, \Gamma \bm x, 0}(t)} \vla \Big/ I_{\infty, \Gamma \bm x, 0}.
\end{equation*}
Noting that the ideal \eqref{Ix ideal} is invariant under the action of $\Gamma$, \emph{i.e.} $\Gamma . I_{\infty, \Gamma \bm x, 0} = I_{\infty, \Gamma \bm x, 0}$, we obtain a well-defined action of $\Gamma$ on the quotient $\U_{\infty, \Gamma \bm x, 0}(\vla)$. We denote the corresponding subalgebra of $\Gamma$-invariants by
\begin{equation} \label{Gamma L vla}
\U^\Gamma_{\infty, \bm x, 0}(\vla) := \big( \U_{\infty, \Gamma \bm x, 0}(\vla) \big)^\Gamma.
\end{equation}

\begin{prop} \label{prop: Lie embedding}
There is an embedding of Lie algebras
\begin{equation} \label{iota map}
\bm \iota : \U^\Gamma_{\infty, \bm x, 0}(\vla) \longhookrightarrow \U(\vla)_{\infty, \bm x, 0}.
\end{equation}
\begin{proof}
Taking the Laurent expansion of an element of $\Lie_{\C_{\infty, \Gamma \bm x, 0}(t)} \vla$ at $0$, $\infty$ and the points in $\bm x$ yields an embedding of Lie algebras
\begin{align} \label{iota map def}
\bm \iota : \Lie_{\C_{\infty, \Gamma \bm x, 0}(t)} \vla &\longhookrightarrow \U(\vla^{\rm op} )_{\infty} \oplus \bigoplus_{i=1}^p \U(\vla)_{x_i} \oplus \U(\vla)_0\\
a f &\longmapsto \big( - \rho(a \otimes \iota_{t^{-1}} f), \rho(a \otimes \iota_{t - x_1} f), \ldots, \rho(a \otimes \iota_{t - x_p} f), \rho(a \otimes \iota_t f) \big). \notag
\end{align}
Note that under the embedding \eqref{iota map def} we have $\bm \iota \big( \sum_{\alpha \in \Gamma} \frac{\cent}{t - \alpha x} \big) = \cent(-1)_x - T \cent(-1)_\infty$ for every $x \in \bm x$ and $\bm \iota \big( \frac{T \cent}{t} \big) = T \cent(-1)_0 - T \cent(-1)_\infty$ from which we deduce that the ideal $I_{\infty, \Gamma \bm x, 0} \subset \Lie_{\C_{\infty, \Gamma \bm x, 0}(t)} \vla$ defined in \eqref{Ix ideal} is mapped to the ideal $I_{\infty, p, 0} \subset \U(\vla^{\rm op} )_{\infty} \oplus \bigoplus_{i=1}^p \U(\vla)_{x_i} \oplus \U(\vla)_0$ defined in \eqref{Ip ideal}. By \cite[Lemma 2.9]{VY2} it follows that \eqref{iota map def} induces an embedding of quotient Lie algebras
\begin{equation*}
\bm \iota : \U_{\infty, \Gamma \bm x, 0}(\vla) \longhookrightarrow \U(\vla^{\rm op})_{\infty} \oplus \bigoplus_{i=1}^p \U(\vla)_{x_i} \oplus \U(\vla)_0 \bigg/ I_{\infty, p, 0}.
\end{equation*}
Finally, by restricting the latter to the subalgebra of $\Gamma$-invariants \eqref{Gamma L vla} we obtain the desired Lie algebra embedding \eqref{iota map}.
\end{proof}
\end{prop}

By virtue of Proposition \ref{prop: Lie embedding}, if $\M_{x_i}$, $i =1, \ldots, p$ is a collection of left modules over $\U(\vla)_{x_i}$ and if $\M_0$ and $\M_\infty$ are left modules over $\U(\vla)_0^\Gamma$ and $\U(\vla^{\rm op})^\Gamma_\infty$ respectively, or alternatively $\M_\infty$ is a \emph{right} module over $\U(\vla)^\Gamma_\infty$, then the tensor product $\M_\infty \otimes \bigotimes_{i=1}^p \M_{x_i} \otimes \M_0$ becomes a left module over the global Lie algebra $\U^\Gamma_{\infty, \bm x, 0}(\vla)$. In particular, we may form the space of coinvariants
\begin{equation*}
\M_\infty \otimes \bigotimes_{i=1}^p \M_{x_i} \otimes \M_0 \bigg/ \U^\Gamma_{\infty, \bm x, 0}(\vla).
\end{equation*}
Many of the statements and proofs of \cite{VY2} concerning cyclotomic coinvariants which did not include the point at infinity can be seen to carry over with minor modifications to the present case.

\subsection{$Y_W$-map}

A left module $M$ over $\U(\vla)$ is said to be \emph{smooth} if for all $a \in \vla$ and $v \in M$ we have $a(n) v = 0$ for all $n \gg 0$. Likewise, we will say that a right module $N$ over $\U(\vla)$ is \emph{co-smooth} if for all $a \in \vla$ and $\eta \in N$ we have $\eta \, a(n) = 0$ for all $n \ll 0$.

Let $M_\infty$ be any right co-smooth module over $\U(\vla)^\Gamma$. We define the quasi-module map
\begin{equation*}
Y_W(\cdot, u) : \ueva \longrightarrow \text{Hom}\big(\M_\infty, \M_\infty((u^{-1}))\big), \qquad
A \longmapsto \sum_{n \in \Z} A^W_{(n)} u^{-n-1}
\end{equation*}
where $u$ is a formal variable and $A^W_{(n)}$ are endomorphisms of $\M_\infty$ for each $A \in \ueva$ and $n \in \Z$, by direct analogy with the definition given in \cite{VY2} for left smooth modules over $\U(\vla)^\Gamma$ as follows. For any $a \in \vla$ we set
\begin{equation*}
Y_W(a(-1)\vac, u) := \sum_{\alpha \in \Gamma} \sum_{n \in \mathbb{Z}} (R_\alpha a)(n) (\alpha u)^{-n-1}.
\end{equation*}
Moreover, the map is defined recursively for any other state in $\ueva$ by letting
\begin{equation} \label{YW rec}
Y_W(a(-1) B, u) := \; \nord{Y_W(a(-1) \vac, u) Y_W(B, u)} + \sum_{\alpha\in\Gamma \setminus \{ 1 \}} \sum_{n \geq 0} \frac{1}{((\alpha - 1) u)^{n+1}} Y_W\big( (R_\alpha a)(n) B, u \big),
\end{equation}
for all $a \in \vla$ and $B \in \ueva$.
Here $\nord{A(u) B(u)} \; = A(u)_+ B(u) + B(u) A(u)_-$ denotes the usual normal ordering where for $A(u) = \sum_{n \in \mathbb{Z}} A(n) u^{-n-1}$ we define
\begin{subequations}
\begin{align}
\label{YW+} A(u)_+ &= \sum_{n < 0} A(n) u^{-n-1} = \sum_{m \geq 0} A(-m-1) u^m,\\
\label{YW-} A(u)_- &= \sum_{n \geq 0} A(n) u^{-n-1}.
\end{align}
\end{subequations}

The following is a direct analogue of \cite[Proposition 3.6]{VY2} for the point at infinity, providing an alternative definition the above $Y_W$-map on co-smooth modules using co-invariants.

\begin{prop}
Let $\M_\infty$ be a right co-smooth module over $\U(\vla)^\Gamma$. We have
\begin{equation} \label{Y-map swap inf}
- \iota_{u^{-1}} [\atp u A \otimes \atp \infty {m_\infty} \otimes \cdots] 
= \big[ \atp \infty {m_\infty} Y_W(A, u) \otimes \cdots \big],
\end{equation}
for all $A\in \ueva$ and $m_\infty \in \M_\infty$.
\begin{proof}
We use induction on the depth of $A$. When $A = \vac$ the result follows from the analogue of \cite[Proposition 3.2]{VY2} including the point at infinity. For the inductive step, we assume that \eqref{Y-map swap inf} holds for states of depth strictly less than that of $A$. Without loss of generality we can take the state $A$ to be of the form $A = a(-1) B$ for some $a \in \vla$ and $B \in \ueva$. 

Let $\M_{x_i}$, $i = 1, \ldots, p$ be any collection of $p \in \Z_{\geq 0}$ modules attached to the points $x_i \in \C$ for $i =1,\ldots,p$ which may include the origin. We write $\bm x := \{ x_i \}_{i=1}^p$ and let $\bm m \in \bigotimes_{i=1}^p \M_{x_i}$.
By definition of the space of coinvariants we have
\begin{equation*}
\big[ f(t) \on (\atp u B \otimes \atp \infty {m_\infty} \otimes \atp{\bm x}{\bm m}) \big] =0, \quad\text{where}\quad 
f(t) = \sum_{\alpha \in \Gamma} \frac{\alpha^{-1} R_{\alpha} a}{\alpha^{-1} t - u}.
\end{equation*}
In other words,
\begin{equation*}
\big[ \iota_{t - u} f(t) \atp u B \otimes \atp \infty {m_\infty} \otimes \atp{\bm x}{\bm m} \big]
- \big[ \atp u B \otimes \atp \infty {m_\infty} \big( \iota_{t^{-1}} f(t) \big) \otimes \atp{\bm x}{\bm m} \big]
+ \sum_{i=1}^p \big[ \atp u B \otimes \atp \infty {m_\infty} \otimes \iota_{t - x_i} f(t) \atp{\bm x}{\bm m} \big] = 0.
\end{equation*}
Thus the left hand side of \eqref{Y-map swap inf} may be written as
\begin{align*}
- &\iota_{u^{-1}} [\atp u{a(-1) B} \otimes \atp \infty{m_\infty} \otimes \atp{\bm x}{\bm m}]
= - \iota_{u^{-1}} \bigg[  {\sum_{\alpha \neq 1} \sum_{n \geq 0} \frac{(R_{\alpha} a)(n)}{((\alpha - 1) u)^{n+1}} \atp uB} \otimes \atp \infty{m_\infty} \otimes \atp{\bm x}{\bm m} \bigg]\\
&- \iota_{u^{-1}} \bigg[ \atp uB \otimes \atp \infty{m_\infty} Y_W(a(-1)\vac, u)_+ \otimes \atp{\bm x}{\bm m} \bigg]
- \iota_{u^{-1}} \bigg[ \atp uB \otimes \atp \infty{m_\infty} \otimes \sum_{i=1}^p \sum_{\alpha \in \Gamma} \sum_{n \geq 0} \frac{(R_{\alpha} a)(n)_{x_i}}{(\alpha u - x_i)^{n+1}} \atp{\bm x}{\bm m} \bigg].
\end{align*}
The states at the point $u$ on the right hand side are of lower depth in $\ueva$, so that we may apply the inductive hypothesis to obtain
\begin{align*}
- &\iota_{u^{-1}} [\atp u{a(-1) B} \otimes \atp \infty{m_\infty} \otimes \atp{\bm x}{\bm m}]
= \sum_{\alpha \neq 1} \sum_{n \geq 0} \frac{1}{((\alpha - 1) u)^{n+1}} \bigg[ \atp \infty{m_\infty} Y_W\big( (R_{\alpha} a)(n) B, u \big) \otimes \atp{\bm x}{\bm m} \bigg]\\
&+ \bigg[ \atp \infty{m_\infty} Y_W(a(-1)\vac, u)_+ Y_W(B, u) \otimes \atp{\bm x}{\bm m} \bigg]
+ \iota_{u^{-1}} \bigg[ \atp \infty{m_\infty} Y_W(B, u) \otimes \sum_{i=1}^p \sum_{\alpha \in \Gamma} \sum_{n \geq 0} \frac{(R_{\alpha} a)(n)_{x_i}}{(\alpha u - x_i)^{n+1}} \atp{\bm x}{\bm m} \bigg].
\end{align*}
Here we made use of the notation \eqref{YW+}. The first two terms on the right hand side are already in the desired form. Taking the map $\iota_{u^{-1}}$ explicitly in the remaining term we may rewrite it as
\begin{align} \label{swap 0 xi}
&\bigg[ \atp \infty {m_\infty} Y_W[B, u] \otimes \sum_{i=1}^p \sum_{\alpha \in \Gamma} \sum_{m \geq 0} \sum_{n=0}^m \binom{m}{n} \frac{x_i^{m-n}}{(\alpha u)^{m+1}} (R_{\alpha} a)(n)_{x_i} \atp{\bm x}{\bm m} \bigg]. 
\end{align}
Now consider the following identity
\begin{equation*}
\sum_{m \geq 0} u^{-m-1} \Big[ g_m(t) \on \big( \atp \infty {m_\infty} Y_W(B, u) \otimes \atp{\bm x}{\bm m} \big) \Big] = 0, \quad\text{where}\quad
g_m(t) = \sum_{\alpha \in \Gamma} \alpha^{-1} (R_{\alpha} a) (\alpha^{-1} t)^m.
\end{equation*}
Using this we may rewrite \eqref{swap 0 xi} simply as
\begin{equation*}
\sum_{m \geq 0} u^{-m-1} \Big[ \atp \infty {m_\infty} Y_W(B, u) \big( \iota_{t^{-1}} g_m(t) \big) \otimes \atp{\bm x}{\bm m} \Big]
= \Big[ \atp \infty {m_\infty} Y_W(B, u) Y_W(a(-1) \vac, u)_- \otimes \atp{\bm x}{\bm m} \Big].
\end{equation*}
Putting the above together and using the recurrence relation \eqref{YW rec} of the $Y_W$-map we obtain
\begin{equation*}
\iota_{u^{-1}} [ a(-1) \atp u B \otimes \atp \infty {m_{\infty}} \otimes \atp{\bm x}{\bm m}] = \big[ \atp \infty {m_{\infty}} Y_W(a(-1) B, u) \otimes \atp{\bm x}{\bm m} \big],
\end{equation*}
as required.
\end{proof}
\end{prop}

Recall that the vacuum Verma module $\ueva$ has the structure of a vertex algebra and so is in particular a vertex Lie algebra. We may thus form the associated Lie algebra $\U(\ueva)$ as well as its subalgebra of $\Gamma$-invariants $\U(\ueva)^\Gamma$.
The following is an immediate generalisation of \cite[Proposition 5.8]{VY2}.

\begin{prop} \label{prop: big co-smooth}
Let $M_\infty$ be a co-smooth right module over $\U(\vla)^\Gamma$. There is a well-defined co-smooth right $\U(\ueva)^\Gamma$-module structure on $M_\infty$ given for all $m_\infty \in M_\infty$ by
\begin{equation*}
m_\infty A^\Gamma(n) := m_\infty A^W_{(n)}.
\end{equation*}
\end{prop}

\def\cprime{$'$}
\providecommand{\bysame}{\leavevmode\hbox to3em{\hrulefill}\thinspace}
\providecommand{\MR}{\relax\ifhmode\unskip\space\fi MR }
\providecommand{\MRhref}[2]{%
  \href{http://www.ams.org/mathscinet-getitem?mr=#1}{#2}
}
\providecommand{\href}[2]{#2}

\end{document}